\documentclass[reqno,12pt]{amsart} 
\usepackage{amsmath,amssymb,amsfonts,eucal}
\usepackage{enumerate, xypic, psfrag}
\usepackage{hyperref}

\newtheorem{thm}{Theorem}[section]
\newtheorem{cor}[thm]{Corollary}
\newtheorem{prop}[thm]{Proposition}
\newtheorem{lem}[thm]{Lemma}

\theoremstyle{definition}

\newtheorem{ex}[thm]{Example}

\numberwithin{equation}{section}

\newcommand\PDinsepbase{Proposition V.1.4}	
\newcommand\qlift{Section 4}		
\newcommand\Paffsubnet{Proposition 4.6}	

\renewcommand \phi{\varphi}
\renewcommand \epsilon{\varepsilon}
\newcommand \eset{\varnothing}

\newcommand\cupdot {\mbox{\hspace{.15em}$\cup$\hspace{-.47em}$\cdot$\hspace{.4em}}} 
\newcommand \inv{^{-1}}
\newcommand \textb{\text{\rm b}}
\newcommand \Lat{\operatorname{Lat}}
\newcommand \Latb{\operatorname{\Lat^{\textb}}}
\newcommand \Span{\operatorname{span}}  
\newcommand \clos{\operatorname{clos}}
\newcommand \bcl{\operatorname{bcl}}
\newcommand \rk{\operatorname{rk}}
\newcommand \codim{\operatorname{codim}}
\newcommand \bgr[1]{\langle#1\rangle}
\newcommand \full{^{{}^{{}_{{}_\bullet}}\!}}
\newcommand \downa{\,{\downarrow}\,}

\newcommand\be{\bar e}
\newcommand\bh{\bar h}
\newcommand\bv{\bar v}
\newcommand\bw{\bar w}
\newcommand \bx{\mathbf{x}}	
\newcommand \bz{\mathbf{z}}	
\newcommand\bH{\bar H}
\newcommand\bE{\bar E}

\newcommand \cA{\mathcal{A}}
\newcommand \cB{\mathcal{B}}
\newcommand \cC{\mathcal{C}}
	
\renewcommand \cH{\mathcal{H}}	
\renewcommand \cL{\mathcal{L}}	
\newcommand \cM{\mathcal{M}}

\newcommand \cP{\mathcal{P}}
\newcommand \cS{\mathcal{S}}

\newcommand \bbA{\mathbb{A}}
\newcommand \bbF{\mathbb{F}}
\newcommand \bbP{\mathbb{P}}

\newcommand \bbR{\mathbb{R}}

\newcommand\fF{\mathbf{F}}	

\newcommand \fG{\mathfrak G}
\newcommand \fH{\mathfrak H}

\newcommand \te{{\tilde e}}
\newcommand \tf{{\tilde f}}
\newcommand \tC{{\tilde C}}

\newcommand \tE{{\tilde E}}
\newcommand \tP{{\tilde P}}

\newcommand \he{{\hat e}}
\newcommand \hf{{\hat f}}
\newcommand \hv{{\hat v}}
\newcommand \hw{{\hat w}}
\newcommand \hC{{\hat C}}
\newcommand \hE{{\hat E}}
\newcommand \hF{{\hat F}}
\newcommand \hN{{\hat N}}
\newcommand \hP{{\hat P}}
\newcommand \hS{{\hat S}}
\newcommand\hW{\hat W}

\newcommand\G{{G\full}}

\begin{document}

\pagestyle{myheadings} 
\markleft{\sc Rigoberto Fl\'orez and Thomas Zaslavsky} 
\markright{\sc Biased Graphs.  VI.  Synthetic Geometry} 
\thispagestyle{empty}


\title{Biased Graphs.  VI.  Synthetic Geometry}

\author{Rigoberto Fl\'orez}
\address{The Citadel, Charleston, South Carolina 29409}
\email{\tt rigo.florez@citadel.edu}

\author{Thomas Zaslavsky}
\address{Dept.\ of Mathematical Sciences, Binghamton University, Binghamton, New York 13902-6000}
\email{\tt zaslav@math.binghamton.edu}

\date{\today}

\begin{abstract}
A biased graph is a graph with a class of selected circles (``cycles'', ``circuits''), called balanced, such that no theta subgraph contains exactly two balanced circles.  A biased graph $\Omega$ has two natural matroids, the frame matroid $G(\Omega)$ and the lift matroid $L(\Omega)$, and their extensions the full frame matroid $\G(\Omega)$ and the extended (or complete) lift matroid $L_0(\Omega)$.  In Part IV we used algebra to study the representations of these matroids by vectors over a skew field and the corresponding embeddings in Desarguesian projective spaces.  Here we redevelop those representations, independently of Part IV and in greater generality, by using synthetic geometry.
\end{abstract}



\thanks{We dedicate this paper to the memory of Seyna Jo Bruskin, who during its preparation made our consultations so much more delightful.}
\thanks{Fl\'orez's research was partially supported by a grant from The Citadel Foundation.  Zaslavsky's research was partially assisted by grant DMS-0070729 from the National Science Foundation.}

\maketitle

\tableofcontents

\newpage 

\pagebreak[2]\section*{Introduction} \label{intro}

This paper bridges a few branches of combinatorial mathematics:  matroids, graphs, and incidence geometry.  The root of our branches is biased graphs.  
A biased graph is a graph that has an additional structure which gives it new properties that are yet recognizably graph-like.  Notably, it has two natural generalizations of the usual graphic matroid, which we call its frame and lift matroids.  In Part IV of this series\footnote{To read this paper it is not necessary to know other parts of this series.  We refer to previous parts, specifically Parts I, II, IV, and (once) V \cite{BG1, BG2, BG4, BG5}, by Roman numeral; e.g., Theorem IV.7.1 is Theorem 7.1 in Part IV \cite{BG4}.}   
we studied linear, projective, and affine geometrical representation of those matroids using coordinates in a skew field.
That leaves a gap in the representation theory because there are projective and affine geometries (lines and planes) that cannot be coordinatized by a skew field and there are biased graphs whose matroids cannot be embedded in a vector space over any skew field.  
The reason for that gap is that our representation theory depended on coordinates.  Here we close the gap with an alternative development of geometrical representation of the frame and lift matroids that is free of coordinates.  
The development is purely synthetic: we reconstruct the analytic point and hyperplane representations from Part IV without coordinates and we prove that the synthetic representations, when in Desarguesian geometries, are equivalent to the analytic ones.

Because all projective and affine geometries of rank higher than 3 are Desarguesian, we are not generalizing Part IV---except for planes.  
Notably, there are many biased-graphic matroids that do not embed in a projective geometry; we describe some examples of rank at least 4 in Section \ref{nonprojective}, but we do not give general criteria to decide when a matroid of a particular biased graph has a projective representation.  
(By contrast, in \cite{BGPP} we prove explicit algebraic criteria for representability of matroid of a biased graph of order 3 in a projective plane, although applying our criteria is difficult because not enough is known about ternary rings of non-Desarguesian planes).  
We nevertheless think our synthetic treatment is well justified as an axiomatic treatment of projectively representable biased-graphic matroids.


\pagebreak[2]\section{Graphs, biased graphs, geometry} \label{prelim}

So as not to require familiarity with the many previous parts of this series, we repeat necessary old definitions as well as giving new ones and providing the required background from algebra and projective geometry.

\subsection{Algebra}\label{algebra}\

We denote by $\fF$ a skew field.  Its multiplicative and additive groups are $\fF^\times$ and $\fF^+$.  it multiplies vectors on the left.

\subsection{Projective and affine geometry}\label{geom}\

\newcommand\0{\mathbf0}
\newcommand\by{\mathbf y}

It is not that easy to find a standard reference for the basics of projective and affine geometry other than in the plane.  
We summarize essentials for readers who are not familiar with them; also to fix our terminology and notation.  
We stress that all our geometries are \emph{finitary}, which means that every dependency contains a finite dependency (as we shall explain shortly).

There are two kinds of projective and affine geometries: those defined by coordinates in a skew field, written $\bbP(\fF)$ where $\fF$ is the skew field, and those defined axiomatically, without coordinates, written $\bbP$.  The dimension may be infinite, but since there is no topology the geometry is \emph{finitary}: every dependent set of vectors contains a finite dependent set.  Every coordinatized geometry is also an axiomatic geometry.  The converse is false in general but a fundamental theorem says that every axiomatic geometry of dimension greater than $2$ has coordinates in some skew field.  For this to be meaningful we must define dimension, which in turn requires the notions of independence and span; those are some of the concepts we explain here.  

Now we begin again from the beginning.

\subsubsection{Coordinatized projective geometry}
We begin with one of the several constructions of projective geometries over a skew field $\fF$.  Consider an $\fF$-vector space $\fF^N$, where $N$ is a set; that is, the vector space comes with a coordinate system.  The set of lines in $\fF^N$ is the point set of a projective geometry $\bbP_N(\fF)$, or simply $\bbP(\fF)$.\footnote{We write a subscript because superscript notation customarily implies the dimension is $\#N$.}  A set of projective points is defined to be a (\emph{projective}) \emph{subspace} when the vectors that belong to its points are the nonzero vectors in a subspace of $V(\fF)$.  The \emph{projective geometry} (or \emph{projective space}) $\bbP_N(\fF)$ is defined to be the incidence structure of projective points and subspaces, \emph{incidence} being set containment.  (We say $\bbP_N(\fF)$ is a projective space \emph{over $\fF$.}  This construction of $\bbP_N(\fF)$ is called \emph{central projection} of $\fF^N$.  Each vector $\bx$ has coordinates, written $(x_i)_{i\in N}$.  Two nonzero vectors in $\fF^{N} $ that lie in the same projective point $\he$ have the same coordinates up to a nonzero scalar: $\by = \alpha \bx$ for some $\alpha \in \fF^\times$.  The \emph{homogeneous coordinates} of $\he$ are $[\he] := [\bx] := [x_i]_{i\in N} := \{ \alpha\bx: \alpha \in \fF^\times \}$ (three notations; each has its use).  Note that $[\bx]=[\by]$; it does not matter which vector one chooses from $\he$ to generate the homogeneous coordinates.  

A projective geometry of this kind is called \emph{coordinatized}.  

\subsubsection{Coordinatized affine geometry}
Next is a simple construction of affine geometries over $\fF$.  The translates of the linear subspaces of $\fF^N$ are called \emph{affine subspaces} (or \emph{affine flats}); with these extra subspaces the set $\fF^N$ becomes an \emph{affine geometry} $\bbA^N(\fF)$ or simply $\bbA(\fF)$.  The affine geometry has the same coordinate system as does $\fF^N$.  

There is a second important construction of $\bbA(\fF)$ from a vector space.  Begin with $\fF^N$ with a distinguished coordinate $x_0$ and choose the linear hyperplane $h_0 := \{ \bx \in \fF^N: x_0=0 \}$.  Translate $h_0$ to the affine hyperplane $h_1 := \{ \bx \in \fF^N: x_0=1 \}$ and define the subspaces of $h_1$ to be its intersections with arbitrary vector subspaces $T$ of $\fF^N$.  With these subspaces $h_1$ is an affine space of dimension $\#N-1$.
This affine space is naturally isomorphic to that obtained from $\bbP_N(\fF)$, the central projection of $\fF^N$, by taking the ideal hyperplane $h_\infty$ to be the hyperplane in $\bbP_N(\fF)$ that corresponds to $h_0$.  
The isomorphism is simple:  an affine subspace $T \cap h_1$ corresponds to the projective subspace implied by the vector subspace $T$.  Thus the points of $h_1$, as an affine space, correspond one-for-one to the lines of $\fF^N$ that do not lie in $h_0$; the ideal points of $\bbP_N(\fF)$ correspond to the lines that do lie in $h_0$.

\subsubsection{Axiomatics}
The properties of an axiomatic projective or affine geometry, $\bbP$ or $\bbA$, are the same, except that there are no coordinates and no projective or affine combinations.  Here is a system of axioms for projective geometr, modernized from \cite[Ch.\ II]{Whd}, also found in \cite[Ch.\ I, Axioms A and E\,0]{VY} and \cite[Theorem 2.3(a--c)]{BC}.
\begin{enumerate}[I.]
\item There is a set $\cP$ of points.
\item There is a set $\cL$ of lines, each of which is a subset of $\cP$.
\item Any two points $p,q$ lie on a unique line, written $pq$.
\item Every line has at least three points.
\item For noncollinear points $p,q,r$, if $p' \in qr$ and $q' \in pr$, then $pp' \cap qq' \neq \eset$.
\item[V$'$.] For points $p,q,r,s$, if $pq \cap rs \neq \eset$, then $pr \cap qs \neq \eset$.
\end{enumerate}
Axioms V and V$'$ are equivalent alternatives; some prefer one, some the other.  
We assume our geometries are finitary (defined below in connection with independence).
If there is at most one point, there cannot be a line.  If there is more than one point, there is a line so there are at least three points.  A \emph{subspace} is a point set that is line-closed, i.e., it contains the whole line determined by any two of its points.  
We do not, as is usually done, assume $\cP$ and $\cL$ are nonempty, because for technical simplicity we want $\eset$ and a solitary point, which are subspaces, to qualify as projective geometries.  However, there are only those two abnormal cases.  

There are axioms for affine geometry (see \cite[Theorem 2.7(a--c)]{BC}) but we simply define an axiomatic affine geometry as a projective geometry with a hyperplane removed as in Section \ref{proj-aff}.

Two fundamental theorems say that every projective or affine geometry of dimension greater than two (to be defined soon) is Desarguesian, which means it satisfies a certain incidence property called \emph{Desargues' theorem}; and that every Desarguesian projective or affine geometry is coordinatizable, that is, it is constructible from a vector space in the manner just described.  Thus, only projective and affine lines and planes can be non-coordinatizable.

\subsubsection{Internal structure}
We defined subspaces; we list some properties.  We often call a projective or affine subspace a \emph{flat} of its geometry (but a vector subspace is not called a flat).  A projective or affine subspace is itself a projective or affine geometry.  The intersection of subspaces is a subspace; intersection is the \emph{meet operation} $\wedge$ in the subspace lattice: $t \wedge t' := t \cap t'$.  The span, $\Span(S)$, of a  point set $S$ is the smallest subspace that contains it.  Span is the \emph{join operation} $\vee$ in the subspace lattice; that is, $t\vee t\ := \Span(t \cup t')$.  In particular, $p \vee q$ denotes the line through points $p$ and $q$.  A \emph{hyperplane} is a maximal proper subspace; every subspace is the intersection of hyperplanes.  A \emph{relative hyperplane} in a subspace $t$ is a subspace that is a hyperplane of $t$.

A point set $S$ is \emph{independent} if every proper subset spans a properly smaller subspace than $\Span(S)$.  We assume that \emph{independence is finitary}: a set is independent if and only if every finite subset is independent.  (This assumption excludes some geometries, in particular those defined with a topology that admits convergent infinite sequences.)  

A \emph{basis} of a subspace $t$ is a maximal independent subset of $t$.  Every subspace has a basis, by Zorn's Lemma.  A proper subset of a basis for $t$ spans a proper subspace of $t$.  Every basis has the same cardinality; the \emph{dimension} of $t$ is $1$ less than the size of a basis.  The matroid rank, however, is the cardinality of a basis.  A line has dimension $1$; a plane has dimension $2$; etc.  In a projective geometry subspace dimension obeys the modular law:  $\dim t + \dim t' = \dim(t \wedge t') + \dim(t \vee t')$.  In an affine geometry that is true when $t \wedge t' \neq \eset$.

In a Desarguesian projective space $\bbP(\fF)$ constructed by central projection, the dimension of a projective subspace $t$ is $\dim T -1$, where $T$ is the corresponding vector subspace.  In particular, $\dim\bbP_N(\fF) = \#N-1$.  The coordinatized definition of projective dependence of a point set $S \subseteq \bbP(\fF)$ is that some finite subset of $S$ satisfies an equation of the form $\sum_{i=1}^k \lambda_i [\bx_i] = 0$ in homogeneous coordinates, with all $\lambda_i \in \fF^\times$.  (It follows that $S$ is affinely independent if and only if every finite subset is affinely independent.)  This agrees with the axiomatic definition.

In a Desarguesian affine space, using the construction of $\bbA(\fF)$ by translating linear subspaces, the dimension of an affine subspace $t$ is the same as that of its corresponding linear subspace $T$, but the notion of dependence changes.  We say a set $S$ in $\bbA(\fF)$ is affinely independent if no nonempty finite subset $\{\bx_1,\ldots,\bx_k$ satisfies an equation $\sum_{i=1}^k \lambda_i \bx_i = \0$ with nonzero scalars such that $\sum_{i=1}^k \lambda_i = 0$.  A consequence is that while $\dim t = \dim T$, a linear basis $B$ of $T$ is one point short of an affine basis; an affine basis is, for instance, $T \cup \{\0\}$.

\subsubsection{From projective to affine}\label{proj-aff}
We often pass back and forth between related projective and affine geometries.  For one direction, in a projective geometry $\bbP$ choose any hyperplane $h_0$.  The point set $\bbP \setminus h_0$ is then an affine geometry, the affine subspaces being the sets $t \setminus h_0$ for all projective subspaces $t \not\subseteq h_0$.  Projectively independent sets in $\bbP \setminus h_\infty$ become affinely independent sets; the bases of affine subspaces are the bases of the corresponding projective subspaces that are disjoint from $h_\infty$.  From the viewpoint of $\bbA$, $h_\infty$ is called the \emph{ideal} or \emph{infinite hyperplane} of $\bbP$.  Points and subspaces in $h_\infty$ are called \emph{ideal} or \emph{infinite}; the rest are called \emph{ordinary}.

If $\bbP$ is Desarguesian, all hyperplanes are equivalent under projective isomorphisms; thus, we get the same affine geometry (up to isomorphism) no matter which hyperplane we choose.  A precise statement is that $\bbP(\fF) \setminus h_\infty = \bbA(\fF)$, i.e., the affine geometry has the same coordinate skew field.  It is often convenient to choose coordinates so that $h_\infty = \{ [\bx] \in \bbP(\fF) : x_0 = 0 \}$ where $x_0$ is a chosen coordinate.  Then all ordinary points have homogeneous coordinates in which $x_0=1$ and the affine coordinates are the projective coordinates with $x_0$ omitted.

When $\bbP$ is a non-Desarguesian projective plane, the resulting affine plane depends on the choice of $h_\infty$, in the sense that not all such planes need be isomorphic.  This fact does not affect our synthetic approach.

\subsubsection{From affine to projective}\label{aff-proj}
In the opposite direction, given any affine geometry $\bbA$ there is a unique projective geometry $\bbP$ in which $\bbA = \bbP \setminus h_\infty$ for a projective hyperplane $h_\infty$.  

First, lines in $\bbA$ are parallel if they are coplanar but have no common point.  For each parallel class of lines, say $L$, create a new point, $L_\infty$, called an \emph{ideal} or \emph{infinite} point; the points of $\bbA$ are now called \emph{ordinary} points.  
The union of $\bbA$ and these ideal points is a projective geometry $\bbP$, which we call the \emph{projective completion} of $\bbA$.  The set $h_\infty$ of ideal points is a hyperplane in $\bbP$, called the \emph{ideal} or \emph{infinite hyperplane}.  
The subspaces of $\bbP$ are of two kinds: first, the \emph{ordinary} subspaces, which are the projective completions $s_\bbP$ of affine subspaces $s$---that means adding to $s$ all the points $L_\infty$ for which $s$ contains a line in the class $L$---and second, the \emph{ideal} subspaces, which are the intersections $s_\infty \cap h_\infty$ of ordinary subspaces with the ideal hyperplane.  Then the hyperplanes of $\bbP$ are $h_\infty$ and the ordinary hyperplanes.

This construction does not use coordinates.  It is a theorem that all projective geometries arise in this way; the proof is that by deleting a hyperplane (any hyperplane) from a projective geometry $\bbP$ one gets an affine geometry of the same dimension, and by adding ideal points using this construction one recovers the original projective geometry.  This is important, because it means that there is no distinguished ideal hyperplane in $\bbP$ unless we have chosen one, and we can choose any hyperplane to be ``ideal''.

Note that a basis for $\bbA$ is the same as a basis for $\bbP$ that avoids the ideal hyperplane (whichever it is).

It follows from the preceding constructions that an affine geometry can be defined as what results by deleting any hyperplane from a projective geometry.

\subsubsection{Projective duality}
In the axioms of projective geometry, points and hyperplanes occupy symmetric roles; by interchanding the names one gets a new axiom system identical or equivalent to the first.  This means that, given any projective geometry $\bbP$, there is a \emph{dual geometry} $\bbP^*$ in which the points are the hyperplanes of $\bbP$ and a hyperplane is a set of all $\bbP$- hyperplanes that contain a fixed point.  In other words, we can identify the hyperplanes of $\bbP^*$ with the points of $\bbP$; and in fact $(\bbP^*)^* = \bbP$ with this identification.  

A simple way to describe $\bbP^*$ is to say that its lattice of subspaces (remembering that a point is a subspace) is the order dual of that of $\bbP$.  An easy consequence is that duality does not change dimension.

\subsection{Graphs}\label{graphs}\

A graph $\Gamma = (N,E)$, with node set $N = N(\Gamma)$ and edge set $E = E(\Gamma)$, may have multiple edges.  Its \emph{order} is $\#N$.  Edges may be links (two distinct endpoints) or half edges (one endpoint); the notation $e_{uv}$ means a link with endpoints $u$ and $v$ and the notation $e_v$ means a half edge with endpoint $v$.\footnote{The (graph) loops and loose edges that appear in other parts of this series are not needed here because they are not important in projective representation.}  If there are no half edges, $\Gamma$ is \emph{ordinary}.  If there are also no parallel edges, it is \emph{simple}.  We make no finiteness restrictions on graphs.  The \emph{empty graph} is $\eset := (\eset,\eset)$.  The \emph{simplification} of $\Gamma$ is the graph with node set $N(\Gamma)$ and with one edge for each class of parallel links in $\Gamma$, where links are parallel if they have the same endpoints.

For $S\subseteq E$, $N(S)$ means the set of nodes of edges in $S$.  
The subgraph induced by $X \subseteq N$ is notated $\Gamma{:}X := (X,E{:}X)$, where $E{:}X := \{ e \in E: N(e) \subseteq X \}.$  

A \emph{separating node} is a node $v$ that separates one edge from another; i.e., there is a partition $\{A,B\}$ of $E$ such that $N(A)\cap N(B) = \{v\}$.  (For example, a node incident to a half edge and another edge is a separating node.)  A graph is \emph{inseparable} if it is connected and has no separating node; equivalently, for every partition $\{A,B\}$ of $E$, $\#(N(A)\cap N(B)) \geq 2$.  A maximal inseparable subgraph of $\Gamma$ is called a \emph{block} of $\Gamma$.

A \emph{circle} is the edge set of a simple closed path, that is, of a connected graph with valency 2 at every node.  $C_n$ denotes a circle of length $n$.  The set of circles of a graph $\Gamma$ is $\cC(\Gamma)$.  A \emph{theta graph} is the union of three internally disjoint paths with the same two endpoints.  A subgraph of $\Gamma$ \emph{spans} (in the sense of graph theory) if it contains all the nodes of $\Gamma$.  

The number of components of $\Gamma$ is $c(\Gamma)$.  For $S \subseteq E$, $c(S)$ denotes the number of connected components of the spanning subgraph $(N,S)$.

\subsection{Biased graphs and biased expansions}\label{bg}\

\subsubsection{Biased graphs}\label{bgbasics}\

A \emph{biased graph} $\Omega = (\Gamma,\cB)$ consists of an underlying graph $\|\Omega\| := \Gamma$ together with a class $\cB$ of circles satisfying the condition that, in any theta subgraph, the number of circles that belong to $\cB$ is not exactly 2.  Such a class is called a \emph{linear class} of circles and the circles in it are called \emph{balanced circles}.  (The ``bias'' is not precisely defined but it means, in essence, the set of unbalanced circles.)  Another biased graph, $\Omega_1$, is a \emph{subgraph} of $\Omega$ if $\| \Omega_1 \| \subseteq \| \Omega \|$ and $\cB(\Omega_1) = \{ C \in \cB(\Omega) : C \subseteq E(\Omega_1) \}$, i.e., the bias of $\Omega_1$ is the restriction of that of $\Omega$.
A biased graph is \emph{simply biased} if it has no balanced digons.  
A simply biased graph is \emph{thick} if every pair of adjacent nodes supports at least two edges (which implies that they induce an unbalanced subgraph).

In a biased graph $\Omega = (\Gamma,\cB)$, an edge set or a subgraph is called \emph{balanced} if it has no half edges and every circle in it belongs to $\cB$.  Thus, a circle is (consistently with the previous paragraph's definition) balanced if and only if it belongs to $\cB$, and any set containing a half edge is unbalanced.  For $S \subseteq E$, $b(S)$ denotes the number of balanced components of the spanning subgraph $(N,S)$.  $N_0(S)$ denotes the set of nodes of all unbalanced components of $(N,S)$.  
A \emph{full} biased graph has a half edge at every node; if $\Omega$ is any biased graph, then $\Omega\full$ is $\Omega$ with a half edge adjoined to every node that does not already support one.\footnote{The unbalanced (graph) loops that appear in other papers in this series are here replaced by half edges since that does not alter the matroids.  
Loops arise in contraction of gain graphs but in this paper there is no such contraction.}

In a biased graph there is an operator on edge sets, the \emph{balance-closure} $\bcl$,\footnote{Not ``balanced closure''; it need not be balanced.} defined by 
$$
\bcl S := S \cup \{ e \notin S : \text{ there is a balanced circle $C$ such that } e \in C \subseteq S \cup \{e\} \}
$$
for any $S \subseteq E$.  This operator is not an abstract closure since it is not idempotent, but it is idempotent when restricted to balanced edge sets; indeed, $\bcl S$ is balanced whenever $S$ is balanced (Proposition I.3.1).  We call $S$ \emph{balance-closed} if $\bcl S = S$ (although such a set need be neither balanced nor closed).

\subsubsection{Gain graphs and expansions}\label{bexp}\

We define gain graphs by first defining group expansions and their gain functions.  

The \emph{group expansion} $\fG\Delta$ of an ordinary graph $\Delta$ by a group $\fG$, in brief the \emph{$\fG$-expansion} of $\Delta$, is a graph whose node set is $N(\Delta)$ and whose edge set is $\fG \times E(\Delta)$, the endpoints of an edge $(g,e) \in E(\fG\Delta)$ being the same as those of $e$. 
The \emph{projection} $p : E(\fG\Delta) \to E(\Delta)$ maps $(g,e)$ to $e$.  
The \emph{full $\fG$-expansion} $\fG\Delta\full$ is $\fG\Delta$ with a half edge attached to each node.  

Each edge $(g,e)$ can be given an orientation, which is either in the direction of $e$ in $\Delta$, or not.  The mapping $\phi: E(\fG\Delta) \to \fG$, defined by $(g,e) \mapsto g$ in the former case and $g\inv$ in the latter, is the \emph{gain function} of $\fG\Delta$.  From now on we abandon the notation $(g,e)$, which was merely a notational device to define $\phi$.  Instead, we write $e$ for an edge of $\fG\Delta$ and infer direction from the context.  For instance, when we compute the gain of a circle, the circle is assumed to have a direction and its edges are oriented in that direction.
When necessary, we disambiguate the value of $\phi(e)$ on an edge $e_{uv} \in E(\fG\Delta)$ by writing $\phi(e_{uv})$ for the gain in the direction from $u$ to $v$.  
Note that a gain function is not defined on half edges.  

A \emph{gain graph} $\Phi = (\Gamma,\phi)$ is any subgraph of a full group expansion, with the restricted gain function.  (Gain graphs can also be defined without mentioning $\Delta$; see Part I.)  In $\Phi$ every circle $C = e_1e_2 \cdots e_l$ has a gain $\phi(C) := \phi(e_1)\phi(e_2)\cdots \phi(e_l)$; although the gain depends on the direction and initial edge, the important property, whether the gain is or is not the identity, is independent of those choices.  A gain graph thus determines a biased graph $\bgr{\Phi} := (\Gamma,\cB(\Phi))$ where $\cB(\Phi)$ is the class of circles with identity gain.

\emph{Switching} a gain graph means replacing $\phi$ by a new gain function defined by $\phi^\zeta(e_{uv}) := \zeta(u)\inv \phi(e_{uv}) \zeta(v)$, where $\zeta$ is some function $V \to \fG$ and $e_{uv}$ indicates that $e$ has the endpoints $u$ and $v$ and the gain is taken in the direction from $u$ to $v$.  Switching $\Phi$ does not change the associated biased graph.

A \emph{biased expansion} of $\Delta$ is a combinatorial generalization of a group expansion.  It is defined as a biased graph $\Omega$ together with a projection mapping $p: \|\Omega\| \to \Delta$ that is surjective, is the identity on nodes, and has the property that, for each circle $C = e_1 e_2 \cdots e_l$ in $\Delta$, each $i = 1,2,\ldots,l$, and each choice of $\te_j \in p\inv(e_j)$ for $j \neq i$, there is a unique $\te_i \in p\inv(e_i)$ for which $\te_1 \te_2 \cdots \te_i \cdots \te_l$ is balanced.  
We write $\Omega\downa\Delta$ to mean that $\Omega$ is a biased expansion of $\Delta$.  
We call $\Omega$ a \emph{$\gamma$-fold} expansion if each $p\inv(e)$ has the same cardinality $\gamma$; then $\gamma$ is the \emph{multiplicity} of the expansion and we may write $\Omega = \gamma\cdot\Delta$.    
It is easy to prove that a biased expansion of an inseparable graph must be a $\gamma$-fold expansion for some $\gamma$ (\PDinsepbase).  
A biased expansion is \emph{nontrivial} if it has multiplicity greater than $1$.  
A trivial expansion of $\Delta$ is simply $\Delta$ with all circles balanced.  
A simple relationship between biased and group expansions is that in the next result.

\begin{lem}\label{L:subgroupexp}
Let $\Delta$ be an inseparable simple graph that has order at least $3$.  A biased expansion of $\Delta$ that equals $\bgr{\Phi}$ for a gain graph $\Phi$ with gain group $\fG$ is the biased graph $\bgr{\fH\Delta}$ of a group expansion by a subgroup $\fH \leq \fG$, and $\Phi$ is a switching of $\fH\Delta$.
\end{lem}

\begin{proof}
Let $\Phi$ denote the gain graph.  First, we show that $\Phi$ is a group expansion.  Take a circle $C \subseteq \Delta$ and a balanced circle $\tilde C \subseteq \Phi$ that is projected $C$.  Assume $\Phi$ has been switched so $\tilde C$ has all identity gains.  

Suppose $g \in\fG$ is the gain of an edge $\te \in p\inv(C)$ and $e = p(\te)$.  Let $f \in C \setminus e$ and define $\tP$ to be the path in $\Phi$ whose projection is $C \setminus f$ and whose edges are those of $\tC$ except that it has the edge $\te$.  Then there is an edge $\tf$ covering $f$ such that $\tP \cup \{\tf\}$ is balanced, from which we infer that $\phi(\tP)\phi(\tf)=\phi(\tP\tf)=1$.  Since $\phi(\tP) = \phi(\te) = g$, $\phi(\tf) = g\inv$.  We have proved that for every two edges $e,f \in C$, if $g$ is the gain of a covering edge $\te$, then $g\inv$ is the gain of an edge $\tf$.  Letting $e_1,e_2,e_3$ be consecutive edges of $C$, it follows that $g\inv$ is the gain of an edge $\te'$ covering $e$.  We conclude that the set $\fH$ of gains of edges in $p\inv(e)$ is closed under inversion and is the same set for every edge $f$ that is in a circle with $e$.  Inseparability implies that $f$ can be any edge of $\Delta$.

Let $\te_1,\te_2,\te_3$ be edges covering $e_1,e_2,e_3$, with $\te_3$ chosen so that the circle $\tC'$ formed from $\tC$ by using $\te_1,\te_2,\te_3$ instead of whatever edges of $\tC$ cover $e_1,e_2,e_3$ is balanced.  Then $1=\phi(\tC')=\phi(\te_1)\phi(\te_2)\phi(\te_3)$.  Therefore, if $g_1,g_2 \in \fH$, so is $(g_1g_2)\inv$ and consequently $g_1g_2$.  So $\fH$ is closed under multiplication.  We have shown $\fH$ is a subgroup of $\fG$ and, since it is the whole set of gains of edges in the fiber of any edge of $\Phi$, $\Phi=\fH\Delta$.
\end{proof}

Biased expansions are studied in depth in \cite{AMQ}, where it is shown that they are built out of groups and irreducible multiary quasigroups.  That structure is not relevant here.  Biased expansions are important for us because they correspond to certain kinds of complete projective and affine structures, as will be clear from our constructions and results.

\subsubsection{Isomorphisms}\label{isom}\

An \emph{isomorphism} of biased graphs is an isomorphism of the underlying graphs that preserves balance and imbalance of circles.  
A \emph{fibered isomorphism} of biased expansions $\Omega_1$ and $\Omega_2$ of the same base graph $\Delta$ is a biased-graph isomorphism combined with an automorphism $\alpha$ of $\Delta$ under which $p_1\inv(e_{vw})$ corresponds to $p_2\inv(\alpha e_{vw})$.  In this paper \emph{all isomorphisms of biased expansions are intended to be fibered} whether explicitly said so or not.  
A \emph{stable isomorphism} of $\Omega_1$ and $\Omega_2$ is a fibered isomorphism in which $\alpha$ is the identity function (the base graph is fixed pointwise).

\subsection{Matroids}\label{matroids}\

We assume acquaintance with elementary matroid theory as in \cite[Chapter 1]{Oxley}.  
The lattice of closed sets of a matroid $M$ is $\Lat(M)$.  
The matroid of projective (or affine) dependence of a projective (or affine) point set $A$ is denoted by $M(A)$.  
All the matroids in this paper are finitary, which means that any dependent set contains a finite dependent set; equivalently, that any element in the closure of an infinite set $S$ is in the closure of a finite subset of $S$.

\subsubsection{Graphic and biased-graphic matroids}\label{gbgmatroids}\

In an ordinary graph $\Gamma$ there is a closure operator $\clos_\Gamma$ on the edges, associated to the graphic matroid, also called the cycle matroid, $G(\Gamma)$.  The circuits of $G(\Gamma)$ are the circles of $\Gamma$.  We call an edge set $S$ \emph{closed in $\Gamma$} (and we call $(N,S)$ a \emph{closed subgraph} of $\Gamma$) if $S$ is closed in $G(\Gamma)$; that is, if whenever $S$ contains a path joining the endpoints of an edge $e$, then $e \in S$.  
The rank function of $G(\Gamma)$ is $\rk S = \#N - c(S)$ for an edge set $S$.  (A second formula is $\rk S = \#N(S) - c(N(S),S)$.  This gives the rank of all finite-rank subsets even if $N$ is infinite.  
The geometric lattice of closed edge sets is $\Lat\Gamma$.

A biased graph $\Omega$, however, gives rise to two kinds of matroid, both of which generalize the graphic matroid.  (For a gain graph, the matroids of $\Phi$ are those of its biased graph $\bgr{\Phi}$.)  First, the \emph{frame matroid} of $\Omega$ (formerly the ``bias matroid''), written $G(\Omega)$, has for ground set $E(\Omega)$.  A \emph{frame circuit}, that is, a circuit of $G(\Omega)$, is a balanced circle, a theta graph that has no balanced circle, or two unbalanced figures connected either at a single common node or by a path that intersects each figure at one of its endpoints.  (An \emph {unbalanced figure} is a half edge or unbalanced circle.)  
The rank function in $G(\Omega)$ is $\rk S = \#N - b(S)$.  (Another formula is $\rk S = \#N(S) - b(N(S),S)$, which gives the rank of finite-rank subsets when $N$ is infinite.  
The \emph {full frame matroid} of $\Omega$ is $\G(\Omega) := G(\Omega\full)$.  The \emph{frame} of $G(\Omega)$ is the distinguished basis consisting of the half edges at the nodes (even if they are not in $\Omega$).  
The closure operator is 
\begin{equation}
\clos S = 
\begin{cases}
E{:}N_0(S) \cup \bcl S &\text{ in general, and} \\
\bcl S &\text{ if $S$ is balanced};
\end{cases}
\label{E:clos}
\end{equation}
recall that $N_0(S)$ is the set of nodes of unbalanced components of $(N,S)$.
(The two definitions for balanced $S$ agree because then $N_0(S)$ is empty.)  
There are two kinds of closed edge set in $G(\Omega)$: balanced and unbalanced; the lattice of closed edge sets is $\Lat\Omega$ and the meet sub-semilattice of closed, balanced sets is $\Latb\Omega$.  

The archetypical frame matroid, though not presented in this language, was the Dowling geometry of a group \cite{CGL}, which is $\G(\fG K_n)$.

Second, the \emph{extended lift matroid} $L_0(\Omega)$ (also called the \emph {complete lift matroid}), whose ground set is $E(\Omega) \cup \{e_0\}$, that is, $E(\Omega)$ with an extra element $e_0$.  A \emph{lift circuit}, which means a circuit of $L_0(\Omega)$, is either a balanced circle, a theta graph that has no balanced circle, two unbalanced figures connected at one node, two unbalanced figures without common nodes, or an unbalanced figure and $e_0$.  The \emph{lift matroid} $L(\Omega)$ is $L_0(\Omega) \setminus e_0$.  The rank function in $L_0(\Omega)$, for $S \subseteq E(\Omega) \cup \{e_0\}$, is $\rk S = \#N - c(S)$ if $S$ is balanced and does not contain $e_0$, and $\#N - c(S) + 1$ if $S$ is unbalanced or contains $e_0$.  The closure operator in $L_0(\Omega)$ is 
$$
\clos_{L_0} S = \begin{cases} 
\bcl S \text{ if $S$ is a balanced edge set, } \\
\clos_{\|\Omega\|} S \cup \{e_0\} \text{ if $S$ is an unbalanced edge set,} \\
\clos_{\|\Omega\|} (S\cap E) \cup \{e_0\} \text{ if $S$ contains $e_0$.}
\end{cases}
$$  
The closure operator of $L(\Omega)$ is the same but restricted to $E$, so $\clos_L S = \clos_{\|\Omega\|} S$ if $S$ is unbalanced.

When $\Omega$ is balanced, $G(\Omega)=L(\Omega)=$ the graphic matroid $G(\|\Omega\|)$ and $L_0(\Omega)$ is isomorphic to $G(\|\Omega\| \cupdot K_2)$ ($\cupdot$ denotes disjoint union).

A (vector or projective) representation of the frame or lift matroid of $\Omega$ is called, respectively, a \emph{frame representation} (a ``bias representation'' in Part IV) or a \emph{lift representation} of $\Omega$.  An \emph{embedding} is a representation that is injective; as all the matroids in this paper are simple, ``embedding'' is merely a short synonym for ``representation''.  
A frame or lift representation is \emph{canonical} if it extends to a representation of $\G(\Omega)$ (if a frame representation) or $L_0(\Omega)$ (if a lift representation).

\begin{thm}[Theorem IV.7.1]\label{T:thick}
Every frame or lift representation of a thick, simply biased graph of order at least $3$ is canonical.

In particular, a frame or lift representation of a biased expansion of an inseparable graph of order at least $3$ is canonical.
\end{thm}

\begin{proof}
The general part is Theorem IV.7.1.  The case of trivial expansions is standard graph theory.  The case of nontrivial expansions is valid because a nontrivial biased expansion is a thick biased graph.
\end{proof}

The next result explains the significance of canonical representation and the importance of a biased graph's having only canonical representations.  It gives algebraic criteria for representability of matroids of biased expansion graphs in Desarguesian projective geometries (\cite{BGPP} develops analogs for non-Desarguesian planes).

\begin{thm}\label{T:canonicalrep}
Let $\Omega$ be a biased graph and $\fF$ a skew field.  
\begin{enumerate}[{\rm(i)}]
\item The frame matroid $G(\Omega)$ has a canonical representation in a projective space over $\fF$ if and only if $\Omega$ has gains in the group $\fF^\times$.  
\item The lift matroid $L(\Omega)$ has a canonical representation in a projective space over $\fF$ if and only if $\Omega$ has gains in the group $\fF^+$.  
\end{enumerate}
\end{thm}

\begin{proof}
The first part is a combination of Theorem IV.2.1 and Proposition IV.2.4.  The former says that gains in $\fF^\times$ imply the existence of a canonical frame-matroid representation by vectors over $\fF$, which is equivalent (by projection) to a canonical projective representation; that gives sufficiency.  The latter says that a representation of $\G(\Omega)$ by vectors over $\fF$ implies that $\Omega$ can be given gains in $\fF^\times$.  Since a canonical representation of $G(\Omega)$ is by definition a representation of $\G(\Omega)$ (restricted to $\Omega$), we have necessity.

The second part is a combination of Theorem IV.4.1 and Proposition IV.4.3, which are the lift-matroid analogs of the results used for (i).  The matroid $\G(\Omega)$ is replaced by the extended lift matroid $L_0(\Omega)$.
\end{proof}

\subsubsection{Non-projective matroids of biased graphs}\label{nonprojective}\

These theorems are the tools we can use to prove that biased-graphic matroids exist of all ranks $n \geq 4$ that cannot be embedded in any projective geometry $\bbP$.  First consider a simple biased-graphic matroid $M\full = \G(\Omega)$ or $L_0(\Omega)$ (both of which have rank $n := \#N$, which may be infinite).  Since $\bbP$ is Desarguesian, i.e., it has coordinates in a skew field $\fF$, by Theorem \ref{T:canonicalrep} $M\full$ can be embedded in $\bbP$ if and only if there is a gain graph $\Phi$ with gains in $\fF^\times$ or $\fF^+$, respectively.  So, any biased graph of rank at least 4 that cannot be given gains in any group will yield two non-projective matroids.  One such graph is Example \ref{I.5.8} below.  Infinitely many non-gainable biased graphs of any finite or infinite order $n\geq4$ can be constructed by methods in \cite{AMQ}, for instance (for finite $n$) by taking a biased expansion of $C_n$ obtained from an $n$-ary quasigroup that has no factorizations, or (for any infinite or large enough finite $n$) by gluing together biased expansion graphs along edges as in \cite[Section 5, especially Corollary 5.5]{AMQ}.

\begin{ex}[Example I.5.8]\label{I.5.8}
The underlying graph is $2C_4$, which consists of four nodes $v_i$ and four double edges, $e_{i-1,i}$ and $f_{i-1,i}$, for $1\leq i\leq 4$, arranged in a quadrilateral (the subscripts are modulo 4).  The balanced circles are $e_{12}e_{23}e_{34}e_{41}$, $f_{12}f_{23}f_{34}f_{41}$, and $f_{12}f_{23}e_{34}e_{41}$.  Example I.5.8 proves that there are no gains for this biased graph (and furthermore it is minor-minimal with that property).
\end{ex}

If we prefer non-projective matroids $G(\Omega)$ or $L(\Omega)$ of biased graphs without half edges, then using thick biased graphs gives matroids that, by Theorem \ref{T:thick}, have only canonical representations.  All the biased graphs mentioned in the previous paragraph are thick, so even without half edges their matroids are not projective.


\pagebreak[2]\section{Menel{\ae}an and Cevian representations of the frame matroid} \label{mencev}

Section IV.2 developed generalizations of the classical theorems of Menelaus and Ceva by means of frame-matroid representations in Desarguesian projective spaces $\bbP(\fF)$ of gain graphs with gains in the multiplicative group of the skew field $\fF$.  Now we give synthetic generalizations of those representations, which apply to any biased graph and any projective geometry, Desarguesian or not.

We begin by stating the classical results, for which one may refer, e.g., to \cite{Alt} or \cite{Cox}.  Both theorems concern a triangle $ABC$ in the affine plane $\bbA^2(\bbR)$.  We consider the edge lines ${AB}$ et al., points $P\in{AB}$, $Q\in{BC}$, $R\in{CA}$, and the signed distances $\overset{\longrightarrow}{AP}, \overset{\longrightarrow}{BP}$, $\overset{\longrightarrow}{BQ}, \overset{\longrightarrow}{CQ}$, and $\overset{\longrightarrow}{CR}, \overset{\longrightarrow}{AR}$.  (For instance, $\overset{\longrightarrow}{AP}$ and $\overset{\longrightarrow}{BP}$ are positive when $P$ is inside the edge.)  Define $\phi(P):=\overset{\longrightarrow}{BP}/\overset{\longrightarrow}{AP}$, $\phi(Q):=\overset{\longrightarrow}{CQ}/\overset{\longrightarrow}{BQ}$, and $\phi(R):=\overset{\longrightarrow}{AR}/\overset{\longrightarrow}{CR}$.

\begin{thm}[Menelaus]\label{T:Men}
The points $P, Q, R$ are collinear if and only if $\phi(P)\phi(Q)\phi(R) = -1$.
\end{thm}

\begin{thm}[Ceva]\label{T:Cev}
The lines $PC, QA, RB$ are concurrent if and only if $\phi(P)\phi(Q)\phi(R) = +1$.
\end{thm}

To convert these results to gain-graphic form, make a graph with nodes $v_A,v_B,v_C$ and edges $e(P)_{v_Av_B}$, $e(Q)_{v_Bv_C}$, $e(R)_{v_Cv_A}$.  Assign gains $\phi(e(P)_{v_Av_B}):=\phi(P)$, $\phi(e(Q)_{v_Bv_C}):=\phi(Q)$, and $\phi(e(R)_{v_Cv_A}):=\phi(R)$ for Ceva and their negatives for Menelaus.  Then the conditions of both theorems become balance of the triangle $e(P)e(Q)e(R)$ and the theorems become special cases of Theorems IV.2.10 (quoted in Theorem \ref{IV.2.10} below) and IV.2.18 (quoted in Theorem \ref{IV.2.18}), which generalize Menelaus and Ceva, respectively, to all dimensions and to arbitrary collections of points in the lines generated by a projective basis, using coordinates to construct the appropriate gains.  The arbitrary collection of points generates the frame matroid (Theorem IV.2.10) and the dual hyperplanes (defined in Section \ref{cev}) generate that same matroid (Theorem IV.2.18).  
(Less comprehensive higher-dimensional Menelaus theorems appeared previously in \cite{Bold, BN, Klein}.  Less comprehensive multidimensional Ceva theorems appeared in \cite{BN, Klein}.  Boldescu \cite{Bold} has a different generalization of Ceva that is not contained in Theorem IV.2.8.)  
Our objective here is to de-coordinatize the constructions and theorems in Part IV to obtain synthetic results.

We note that Menelaus and Ceva are traditionally stated by comparing the product of three of the distances to the product of the three complementary distances, which obscures their true gain-graphic nature.

\subsection{Menelaus:  points}\label{men}\

Suppose we have a set $\hN$ of independent points in a projective geometry $\bbP$ (which may be non-Desarguesian).  We call an \emph{edge line} any line of the form $pq$ for $p,q \in \hN$.  The union of all edge lines is $\hE\full(\hN)$; without $\hN$ it is $\hE(\hN) = \hE\full(\hN) \setminus \hN$.  The matroid $M(\hE\full(\hN))$, or any submatroid, is called a \emph{frame matroid} with $\hN$ as its \emph{frame}.  (This is an abstract frame matroid, not presented as the frame matroid $G(\Omega)$ of a biased graph.)

Now let us choose a set $\hE \subseteq \hE\full(\hN)$.  
Following \cite{FrM} we shall set up a biased graph $\Omega(\hN,\hE)=(N,E,\cB)$ whose node set $N$ and edge set $E$ are in one-to-one correspondence, respectively, with $\hN$ and $\hE$ but which are disjoint from $\hE\full(\hN)$ and from each other, and such that the full frame matroid $\G(\Omega)$ is naturally isomorphic to the projective frame matroid $M(\hN\cup\hE)$.  We call this representation of the frame matroid of $\Omega$ \emph{Menel{\ae}an} because it generalizes the theorem of Menelaus.

For the underlying graph $\Gamma=(N,E)$ we let $p\in \hE \setminus \hN$ correspond to a link $e$ whose endpoints $v,w \in N$ are chosen so $p\in \hv \hw$.  Then we write $p=\he$.  
(For $v\in N$ or $e\in E$, we write $\hv$ or $\he$ for the corresponding point in $\bbP$.)  A point $p\in \hE \cap \hN$ corresponds to a half edge at that node $v$ for which $p = \hv$.  
For the bias we need the notion of a \emph{cross-flat}.  This is any projective flat that is disjoint from $\hN$.  We define a circle $C$ in $\Gamma$ to be \emph{balanced} if $\hC$ lies in a cross-flat.  
(This differs from the definition of balance used in \cite{FrM}, but \cite[Lemma 4]{FrM} states that both definitions are equivalent.)  
That completes the definition of $\Omega(\hN,\hE)$.

\begin{thm} \label{TD:men}
{\rm(a)} For any $\hE\subseteq \hE\full(\hN)$, $\Omega(\hN,\hE)$ is a biased graph.  The natural correspondence $e \mapsto \he$ is an isomorphism $G(\Omega(\hN,\hE)) \cong M(\hE)$.

{\rm(b)} The closed, balanced sets of $G(\Omega(\hN,\hE))$ are the edge sets corresponding to the intersections of $\hE$ with the cross-flats of $\hN$.
\end{thm}

\begin{proof} 
For a synthetic proof of part (a), we note that it is precisely the statement of \cite[Theorem 1]{FrM}, whose proof is coordinate-free.

The cross-flat property (b) has two parts.  By the matroid isomorphism a subset $S$ of $E$ is closed if and only if $\hS$ is the intersection of $\hE$ with a flat of $\bbP$, and furthermore its closure in $\G(\Omega(\hN,\hE))$ corresponds to the intersection of $\hN \cup \hE$ with the projective closure $\Span\hS$.  Lemma 5 of \cite{FrM} states that $S$ is balanced if and only if $\Span\hS$ does not contain any element of $\hN$; in other words, $\hS$ spans a cross-flat.
\end{proof}

\begin{ex}\label{X:nonuniquegraph}
Notice that the construction of $\Omega(\hN,\hE)$ is carried out with respect to an independent set $\hN$ that is given in advance.  If $\hE$ is small, it may not determine $\hN$ uniquely, and different choices of $\hN$ may lead to biased graphs with different edge structure.  

For example, in a projective plane choose a set $\hE=\{p_1,\ldots,p_6\}$ of 6 points so that no three are collinear.    
Let $q_i:=p_{i}p_{i+1} \wedge p_{i+2}p_{i+3}$ (with subscripts modulo 6) and assume that $\hN:=\{q_1,q_3,q_5\}$ and $\hN':=\{q_2,q_4,q_6\}$ are noncollinear triples. (This is possible unless the plane has very small order; we omit the details.)  
Then $\hN$ and $\hN'$ are two different frames for $\hE$.  The edge set of both underlying graphs, $\Gamma$ and $\Gamma'$, is $E =\{e_1,\ldots,e_6\}$ with $p_i \leftrightarrow e_i$ but the edges are attached differently.  In the frame $\hN$ the node set is $N=\{v_1,v_3,v_5\}$ (the nodes correspond to $q_1,q_3,q_5$); since $p_1,p_2 \in q_1q_3$, $e_1$ and $e_2$ are parallel edges with endpoints $v_1,v_5$.  However, in the frame $\hN'$ with node set $N'=\{v_2,v_4,v_6\}$, $p_1 \in q_4q_6$ and $p_2 \in q_2q_6$ so $e_1$ has endpoints $v_4,v_6$, and $e_2$ has endpoints $v_2,v_6$; the same two edges are not parallel.  
Thus the underlying graphs, hence the biased graphs, have different edge structure with respect to the two different frames, so even ignoring the different node names, they are not the same graph.  
(The matroids in both cases are the same because they are the matroid of $\hE$.  The graphs happen to be isomorphic, but further discussion of that is outside our scope.)
\end{ex}

If $N$ is a set, $K_N$ denotes the complete graph with node set $N$.

\begin{cor} \label{PD:menall} 
$\Omega(\hN,\hE\full(\hN))$ is a full biased expansion of $K_N$.
\end{cor}

\begin{proof}
The projection $p:\Omega(\hN,\hE\full(\hN)) \to K_N$ is obvious.  We have to show that, for any circle $C=e_1 \cdots e_l$ in $K_N$ and any $\te_1 \in p\inv(e_1), \ldots, \te_{l-1} \in p\inv(e_{l-1})$, there is exactly one $\te_l \in p\inv(e_l)$ that makes the circle $\te_1 \cdots \te_l$ balanced.

Let $P=\te_1 \cdots \te_{l-1}$, with endpoints $v$ and $w$ (which are also the endpoints of $e_l$).  Then $\rk(\hP) =l-1$ and $\Span(\hP \cup \hv\hw) = \Span \hN(C)$.  (By $N(C)$ we mean the node set of $C$ and by $\hN(C)$ we mean the corresponding projective point set.)  
So, 
\begin{align*} 
\rk((\Span{\hP}) \wedge \hv \hw) 
&= \rk(\hP) + \rk(\hv \hw) - \rk((\Span{\hP}) \vee \hv \hw) \\
&= (l-1) + 2 - \rk(\hP \cup \hv \hw ) = 1.
\end{align*}
That is, there is a unique point $q\in (\Span{\hP}) \cap \hv\hw$.  This point cannot be in $\hN$, for $(\Span{\hP}) \cap \hN = \eset$ by balance of $P$.  Therefore $q=\hf $ for a link $f \in p\inv(e_l)$.  Note that we have proved $\Span{\hP}$ to be a cross-flat.  

We now have a circle $P \cup \{f\}$, and we know that $\Span(\hP \cup \{\hf\}) = \Span{\hP}$, a cross-flat.  Therefore $P \cup \{f\}$ is balanced, so we may take $\te_l = f$.  There can be no other choice for $\te_l$, because $q$ is the unique point in $(\Span{\hP}) \cap \hv \hw$. 
\end{proof}

Finally, we show that our synthetic Menel{\ae}an representation coincides with the analytic one presented in Section IV.2.5 when the latter is defined.  
Suppose $\Phi=(N,E,\phi)$ is a gain graph with gains in $\fF^\times$.  
In a coordinatized projective space $\bbP_N(\fF)$ (derived from $\fF^N$) there is the natural basis $\hN := \{\hv=[0,\ldots,0,1_v,0,\ldots,0]: v \in N \}$.  Choose the hyperplane $\sum_{v\in N} x_v=0$ to be the ideal hyperplane $h_\infty$.  
The analytic definition of a Menel{\ae}an representation of $\Phi$, from Equation (IV.2.3), is 
\begin{align*} 
\he &= \hv \quad\text{ for a half edge at $v$,} \\
\he &= 
  \left.\begin{cases} 
\dfrac {1}{1-\phi(e_{vw})}(\hv - \phi(e_{vw}) \hw) &\text{if } \phi(e_{vw}) \neq 1,  \\[10pt] 
h_\infty \wedge \hv\hw &\text{if } \phi(e_{vw}) = 1
  \end{cases} \right\}\text{ for a link $e=e_{vw}$.}
\end{align*}
Then $\cM[\Phi] = \{ \he : e \in E \}$ is the (analytic) projective Menel{\ae}an representation of $G(\Phi)$.  

\begin{thm}\label{T:menequiv}
Consider a gain graph $\Phi=(N,E,\phi)$ with gains in the multiplicative group of a skew field $\fF$, and the Desarguesian projective space $\bbP_N(\fF)$ coordinatized by $\fF$.  
Then $\cM[\Phi]$ is a synthetic Menel{\ae}an representation of $G(\Phi)$.  Stated precisely, the mappings $v \mapsto \hv$ and $e \mapsto \he$ are isomorphisms $\bgr{\Phi} \cong \Omega(\cM[\Phi])$ and $G(\Phi) \cong M(\cM[\Phi])$.
\end{thm}

\begin{proof}
If we show that $\Omega(\cM[\Phi])$ is naturally isomorphic to $\bgr{\Phi}$ (by the correspondences $\hv \leftrightarrow v$ and $\he \leftrightarrow e$), then the matroids are also naturally isomorphic.  The underlying graphs are clearly isomorphic.  We have to show that corresponding edge sets are balanced in both.  
We employ the Generalized Theorem of Menelaus from Part IV, which we quote:

\begin{thm}[Theorem IV.2.10]\label{IV.2.10}
{\rm(a)} The set of flats spanned by $\cM[\Phi]$ is isomorphic to $\Lat G(\Phi)$ under the isomorphism induced by $e \mapsto \be$.
{\rm(b)} If $S \subseteq E$ and $W \subseteq N$, the flat generated by the points $\be$ for $e \in S$ contains $\Span(\hat{W})$ if and only if $W \subseteq N_0(S)$.
\end{thm}

In $\Omega(\cM[\Phi])$ an edge set $\hS$ is balanced if and only if its span is a cross-flat, i.e., it contains none of the basis points $\hv \in \hN$.  Theorem \ref{IV.2.10}(b) implies that a flat spanned by $\hS \subseteq \cM[\Phi]$, corresponding to $S \subseteq E(\Phi)$, is a cross-flat if and only if $S$ is balanced.  This is the same criterion for balance as that in Theorem \ref{TD:men}, given that an edge set is balanced if and only if its matroid closure is balanced 
and that, by Theorem \ref{IV.2.10}(a), $\cM[\Phi]$ represents $G(\Phi)$.  We have shown that the criterion for balance in $\Omega(\cM[\Phi])$ is equivalent to that in $\Phi$; consequently, $\Omega(\cM[\Phi])$ and $\bgr{\Phi}$ are naturally isomorphic.
\end{proof}

A special case of Theorem \ref{T:menequiv} is that, when $\bbP$ is Desarguesian over a skew field $\fF$, the biased graph $\Omega(\hN,\hE\full(\hN))$ of Corollary \ref{PD:menall} is the biased graph of the full group expansion $\fF^\times K_N\full$.  

A more general way biased expansions appear is set out in the next result.
Let $\Delta\subseteq K_N$ be a simple graph on node set $N$.  By an \emph{edge line of $\Delta$} we mean an edge line $l_e$ of $\hN$ that corresponds to an edge $e$ of $\Delta$.  Let $\hE(\Delta)$ be the union of the edge lines of $\Delta$.  
Call $\hE\subseteq \hE(\hN)$ \emph{cross-closed with respect to $\Delta$} if for any cross-flat $t$ generated by points in $\hE$, $t \cap \hE(\Delta) \subseteq \hE$.  
(Equivalently, for any edge line $l_e$ of $\Delta$, $t \cap l_e \subseteq \hE$.)

\begin{cor}  \label{CD:mencross}
{\rm(I)}  Suppose $\hE\subseteq\hE(\hN)$ is contained in the union of edge lines of $\Delta$.  Then $\hE$ is cross-closed with respect to $\Delta$ if and only if $\Omega(\hN,\hE)$ is a biased expansion of a closed subgraph $\Delta'$ of $\Delta$. 

{\rm(II)}  Suppose further that $\Delta$ is inseparable and has order at least $3$ and that $\bbP$ is coordinatized by a skew field $\fF$.  Then $\hE$ is cross-closed with respect to $\Delta$ if and only if $\Omega(\hN,\hE) = \bgr{\fH\Delta'}$, where $\Delta'$ is some closed subgraph of $\Delta$ and $\fH$ is some subgroup of $\fF^\times$.
\end{cor}

\begin{proof}  
We prove (I).  Let $\tE'$ denote the edge set of $\Omega(\hN,\hE)$.  

To prove necessity, suppose $\hE$ is cross-closed with respect to $\Delta$.  
We identify $\Delta'$ by its edge set $E' := \{ e \in E(\Delta) : \hE \cap l_e \neq \eset \}$.
If $t$ is a cross-flat, we identify a subgraph $\Delta_t$ (with node set $N$) by its edge set $E_t := \{ e \in E(\Delta) : t \cap l_e \cap \hE \neq \eset \}$.  Evidently, $\Delta' = \bigcup_t \Delta_t$, where the union is taken over all cross-flats generated by finite subsets of $\hE$.

The first step is to prove $E'$ is closed in $\Gamma$.  
Let $P=e_1e_2\cdots e_k$ be a path in $\Delta'$ with edges $e_i=e_{v_{i-1}v_i}$, whose endpoints $v_0,v_k$ are joined by an edge $e$ of $\Delta$ so that $e \in \clos P$, and choose a point $q_i \in l_{e_i} \cap \hE$ for $i=1,\ldots,k$.  Define $t:=\Span\{q_1,\ldots,q_k\}$; $t$ is a cross-flat because its dimension is (at most) $k-1$ but if it contained any point $\hv$ it would span all of $\hv_0,\hv_1,\ldots,\hv_k$ and have dimension $k$.  Then $P \subseteq E_t$ and $t \wedge l_{e_i} = q_i$.  Define $u:=\Span\{\hv_0,\ldots,\hv_k\}$, of dimension $k$.  $t$ is a proper subspace of $u$ with dimension $k-1$ because $t \vee \hv_0 = u$.  The edge line $l_e$ is a 1-dimensional subspace of $u$.  Thus $t \cap l_e$ is a point, which is in $\hE$ by cross-closure; it follows that $e \in E_t \subseteq E'$.  This proves that $E'$ is closed in $\Delta$.

Now we prove that $\Omega(\hN,\hE)$ is a biased expansion of $\Delta'$; it is similar to the preceding but more involved.  
Suppose we have $P$ and $e$ in $E'$ as before and arbitrary edges $\te_i \in p\inv(e_i)\cap \tE'$ for $i=1,\ldots,k$.  These edges correspond to points $q_i \in \hE \cap l_{e_i}$.  The flat $t$ spanned by $q_1,\ldots,q_k$ is a cross-flat that intersects $l_e$ in a point $q_e$, and by cross-closure $q_e \in \hE$.  In $\Omega(\hN,\hE)$, $q_e$ corresponds to an edge $\te \in p\inv(e)$ that is in $\tE'$ because $q_e\in\hE$ and that forms a balanced circle with $\te_1,\ldots,\te_k$ because $t$ is a cross-flat.  Clearly, $q_e$ is uniquely determined by $q_1,\ldots,q_k$, so the requirement for a biased expansion that a unique edge in the fiber $p\inv(e)$ forms a balanced circle in $\Omega(\hN,\hE)$ is satisfied.  We conclude that $\Omega(\hN,\hE) \downa \Delta'$.

For sufficiency suppose $\hE$ has the property that $\Omega(\hN,\hE)$ is a biased expansion of a closed subgraph $\Delta' \subseteq \Delta$.  
Let a point set $Q \subseteq \hE$ span a cross-flat $t$ such that $t \cap l_e$, for some $e \in E(\Delta)$, is a point $q \in \hE$ but not in $t$; we show this leads to a contradiction.  
We may assume $Q$ is a minimal subset of $t$ that spans $q$, so $Q \cup \{q\}$ is a minimal dependent set and thus finite.  Let $Q = \{q_1,\ldots,q_k\}$, let $q_i$ correspond to $\te_i \in \hE$, and let $e_i=p(\te_i)$; then $e_i \in E(\Delta')$ by the definition of $\hE$.  Since $q\in l_e$, $q$ corresponds to some edge $\te \in \tE'$ such that $p(\te) = e$.  It follows (since $t \cap \hN = \eset$) that $\{\te_1,\ldots,\te_k,\te\}$ is a frame circuit in $\Omega(\hN,\hE)$.  It is also balanced because $Q \cup \{e\}$ lies in the cross-flat $t$; therefore it is a balanced circle in $\Omega(\hN,\hE)$.  The projection of a balanced circle in $\Omega(\hN,\hE)$ is a circle in $\Delta$.  Thus, $e \in E(\Delta')$ since  $\Delta'$ is closed and all $e_i \in E(\Delta')$.

Now we have $\te_1,\ldots,\te_k,\te$ forming a balanced circle in $\Omega(\hN,\hE(\Delta))$.  The edges $\te_1,\ldots,\te_k$ make a path in $\Omega(\hN,\hE)$ whose projection, with $e$, makes a circle in $\Delta'$.  By the definition of a biased expansion of $\Delta$, there is only one edge in $\Omega(\hN,\hE(\Delta))$ that forms a balanced circle with $\te_1,\ldots,\te_k$; $\te$ is that edge.  By the definition of a biased expansion of $\Delta'$, there is an edge in $\tE'$ that forms a balanced circle with $\te_1,\ldots,\te_k$; $\te$ must be that edge because $\Omega(\hN,\hE) \subseteq \Omega(\hN,\hE(\Delta))$.  Therefore, $\te \in \tE'$, and it follows that $q \in \hE'$.

Part (II) follows from part (I) and Lemma \ref{L:subgroupexp}.
\end{proof}

\subsection{Ceva:  hyperplanes}\label{cev}\

In the dual representation (that is projective duality, not matroid duality), which we call \emph{Cevian} because it generalizes Ceva's theorem, an edge corresponds to a projective hyperplane instead of a point.  

We think the frame representation by hyperplanes rather than points is the more natural because of the simpler form $x_j=x_i\phi(e_{ij})$ of the representation formula for gain graphs in a Desarguesian projective geometry with coordinates $x_k$ that correspond to the nodes (as shown in Sections IV.2.5 vs.\ IV.2.6).  
Nevertheless, a biased graph has both representations if it has either one, say Menel{\ae}an in $\bbP$ and Cevian in the dual geometry $\bbP^*$, since the Menel{\ae}an representation in $\bbP$ dualizes to a Cevian representation in $\bbP^*$ with the same matroid structure, and vice versa.

A point basis $\hN$ of $\bbP$ generates a hyperplane basis $\hN^* := \{ p^* : p \in \hN\}$ where $p^* := \Span( \hN \setminus p ).$  (A hyperplane basis is a set of hyperplanes whose intersection is empty and that are independent, so that none contains the intersection of any others.)  A hyperplane basis is a point basis of the dual geometry $\bbP^*$; the concepts are equivalent.  For $p$ in an edge line $\hv\hw$ of $\hN$, let 
\begin{equation}
h(p) := \begin{cases}
\Span(\hN \setminus \{\hv,\hw\} \cup \{p\})	&\text{if } p \notin \hN, \\
p^* := \Span(\hN \setminus p)	&\text{if } p \in  \hN.
\end{cases}
\label{E:cevianhyp}
\end{equation}
We call this an \emph{apical hyperplane} of $\hN$ with $p$ as \emph{apex}.  In particular, $\hv^*$ is a \emph{node hyperplane}.  Let $\hE^*{}\full(\hN)$ be the set of all apical hyperplanes, including those with apices in $\hN$, and let $\hE^*(\hN)$ be the set of those whose apices are not in $\hN$; thus, $\hE^*{}\full(\hN) = \hE^*(\hN) \cup \hN^*$.  

Any set $\cA$ of hyperplanes has a matroid structure given by, for example, the rank function $\rk\cS := \codim\bigcap\cS$ for $\cS\subseteq\cA$.  (It is not the same as the matroid of the apices; see Example \ref{X:cevianapices}.)

In the Cevian representation of a frame matroid of a biased graph an edge $e_{vw}$ will correspond to a hyperplane of the form $\bh(e):=h(\be)$ for some $\be \in \hv\hw\setminus \{\hv,\hw\}$.  Let $\cA$ be any subset of $\hE^*{}\full(\hN)$; we call $\cA$ a \emph{Cevian arrangement} of hyperplanes.  
We show it represents $G(\Omega)$ for a biased graph $\Omega$ constructed from $\cA$, by which we mean that the matroid of $\cA$ is naturally isomorphic to $G(\Omega)$, or equivalently that the intersection lattice $\cL(\cA) := \{ \bigcap\cS : \cS\subseteq\cA \}$, which is partially ordered by reverse inclusion, corresponds isomorphically to $\Lat \Omega$ and that the rank function in the latter equals codimension in $\cL(\cA)$.  

The setup is similar to that of Section \ref{men}.  The underlying graph $\Gamma=(N,E)$ has node set $N$ in bijection with $\hN$ and $\hN^*$ by $v \leftrightarrow \hv \leftrightarrow \hv^*$.  The edge set is in bijection with $\cA$ by the apical hyperplane function $\bh: E \to \cA$.  A hyperplane $h(\be)$ with apex $\be\in\hv\hw\setminus\hN$ becomes a link $e_{vw}$.  For a hyperplane $\hv^*$, the edge is a half edge $e_v$.  The intended rank function of this graph is easy to state:  
\begin{equation}
\rk_\bbP(S) := \codim \bigcap \bh(S), 
\label{E:cevrank}
\end{equation}
where $S\subseteq E$ and $\bigcap \bh(S) := \bigcap_{e \in S} \bh(e)$.  The class of balanced circles should be the class $\cB := \{ C \in \cC(\Gamma): \rk_\bbP C < \#C \}$.   In order to define a biased graph $\Omega(\cA)$, $\cB$ must be a linear class: no theta subgraph can contain exactly two circles in $\cB$.   Moreover, $\cA$ represents $G(\Omega(\cA))$ if and only if the rank function defined by Equation \eqref{E:cevrank} agrees with the rank function of the frame matroid.  Thus, we have to prove $\cB$ is a linear class and that \eqref{E:cevrank} gives the correct rank.

Some necessary notation:  For an intersection flat $t \in \cL(\cA)$, define $\cA(t) := \{ h(\be) \in \cA: h(\be) \supseteq t \}$.  This implies a mapping $\bH$ from $\cL(\cA)$ to the power set $\cP(E)$ by $t \mapsto \bh\inv(\cA(t))$, the edge set that naturally corresponds to $\cA(t)$.  

We also need a formula for deletion of a pendant edge; that is, a link that has an endpoint with no other incident edges.

\begin{lem}[Pendance Reduction]\label{L:pendant}
An edge set $S$ that has a pendant edge $e$ satisfies $\rk_\bbP S = \rk_\bbP(S\setminus e) +1$.
\end{lem}

\begin{proof}
Let $e=e_{vw}$ with $v$ the endpoint at which $e$ is the only incident edge and let $W := N(S)$; then $W \setminus v = N(S \setminus e)$ (unless $S = \{e\}$; that case is easy).  It follows that $t_\bbP(S\setminus e) \subset t_\bbP(S))$ since $t_\bbP(S) \subseteq \Span(\hN \setminus (\hW\setminus\hv) \not\supseteq t_\bbP(S\setminus e$.  Thus, $\rk_\bbP S > \rk_\bbP(S\setminus e)$.  But deleting one hyperplane from $\bh(S)$ can only decrease the intersection rank by 1; the lemma follows.
\end{proof}

\begin{thm}\label{T:cev}
For any $\cA \subseteq \hE^*{}\full(\hN)$, $\Omega(\cA)$ is a biased graph.  The natural correspondence $e \mapsto h(\be)$ is an isomorphism $G(\Omega(\cA)) \cong M(\cA)$.

The intersection flats of $\cA$ correspond to the closed sets of $G(\Omega(\cA))$ through the natural correspondence $\bH$.  
A connected edge set $S$ of $\Omega(\cA)$ is unbalanced if and only if $\bigcap \bh(S) \in \cL(\hN^*)$.
\end{thm}

\begin{proof}
There is a proof by projective duality since a Cevian representation in $\bbP$ is a Menel{\ae}an representation in $\bbP^*$.  However, we prefer to give a direct proof using the correspondence between edges and Cevian hyperplanes and the definition of rank as codimension.

It follows from pendance reduction that the rank of a forest $S \subseteq E$ is $\rk_\bbP S = \#S$.  Then it follows that the rank of a circle $C$ is either $\#C$ or $\#C-1$.  Therefore, the balanced circles are those with rank $\rk_\bbP C = \#C-1$.  

Next, we prove that a theta graph, composed of internally disjoint paths $P, Q, R$ with the same endpoints, cannot contain exactly two balanced circles.  Suppose $P\cup Q$ and $Q \cup R$ are balanced.  Let $e\in P$ and $e'' \in R$.  Then $P\cup Q \cup R \setminus \{e,e''\}$, being a tree, has rank $\#P+\#Q+\#R-2$.  
By balance, $\rk_\bbP(P\cup Q) = \rk_\bbP(P\cup Q \setminus e)$ and $\rk_\bbP(Q\cup R) = \rk_\bbP(Q\cup R \setminus e'')$.  It follows that $\bh(P\cup Q) = \bh(P\cup Q \setminus e)$ and $\bh(Q\cup R) = \bh(Q\cup R \setminus e'')$, and consequently that $\bh(P\cup Q \cup R) = \bh(P\cup Q \cup R \setminus \{e,e''\})$.  We conclude that $\rk_\bbP(P\cup Q \cup R) = \#P+\#Q+\#R-2$.  
Let $e'\in Q$.  Then $\rk_\bbP(P\cup Q \cup R \setminus e') \leq \#P+\#Q+\#R-2$, which implies that $\rk_\bbP(P\cup R) \leq \#P+\#R-1$ by pendance reduction.  However, $P \cup R$ is a circle so $rk_\bbP(P \cup R) \geq \#P + \#R - 1$; thus we have equality:  $rk_\bbP(P \cup R) = \#P + \#R - 1$.  That is the definition of balance of $P\cup R$.  It follows that no theta graph can have exactly two balanced circles; i.e., $\Omega(\cA)$ is a biased graph.

We now prove matroid isomorphism by proving that the rank functions $\rk_\bbP$ and $\rk$ (in $G(\Omega)$) agree.  (Since our matroids are finitary, it is only necessary to consider finite sets of edges or hyperplanes.)  First, suppose $(W,S) \subseteq \|\Omega(\cA)\|$.  Clearly, $\bigcap \bh(S) \supseteq \bigcap(\hN^*-\hW^*)$, so $\rk_\bbP S \leq \rk_\bbP\hW^* = \#W$.  Now suppose $(W,S)$ is connected (and $S \neq \eset$).  Then $\rk_\bbP S \geq \#W-1$ because $S$ contains a spanning tree $T$.  If $S$ is balanced, then $S = \bcl T$ by Equation \eqref{E:clos} so that $\rk_\bbP S = \rk_\bbP \bcl T = \rk_\bbP T = \#T = \#W-1 = \rk S$, because $\bcl$ adds edges $e$ to $T$ only through balanced circles $C$, in which $\rk_\bbP C = \#C-1 = \rk_\bbP(C \setminus e)$.  If $S$ is unbalanced, then it contains an unbalanced circle $C_0$ and we can choose $T \supseteq C_0\setminus e$ for an edge $e \in C_0$.  Then $\rk_\bbP(C_0\cup T) = \rk_\bbP C_0 + \#(T\setminus C_0)$ (by pendance reduction) $= \#T+1 = \#W$.  This shows that $\rk_\bbP$ agrees with $\rk$ on finite sets, from which we conclude that $M(\cA)$ and $G(\Omega(\cA))$ are isomorphic under the mapping $\bH$.

It is easy to prove by similar reasoning that an edge set $S$ is closed in $\Omega(\cA)$ if and only if $\bH(S) = \cA(t)$ for the intersection flat $t := \bigcap\bH(S)$.  

Our previous arguments show that, for a connected subgraph $(W,S)$, $\bigcap \bh(S) = \bigcap \hW^*$ if $S$ is unbalanced.  If $S$ is balanced, then its rank is $\#W-1$ so $\bigcap\bh(S)$ is either not the intersection of node hyperplanes $\hv^*$, or is the intersection $\bigcap(\hW^*\setminus\hw^*)$ for some $w \in W$.  If the latter, there is an edge $e \in S$ incident with $w$ and the hyperplane $\bh(e)$ is not contained in $\bigcap(\hW^*\setminus\hw^*)$, a contradiction.  Therefore, $\bigcap \bh(S) \notin \cL(\hN^*)$.
\end{proof}

Theorem \ref{T:cev} has a corollary similar to Corollary \ref{PD:menall}, whose statement and proof we omit.  We cannot omit a proof that our synthetic Cevian representation agrees with the analytic Cevian representation of $G(\Phi)$ developed in Section IV.2.6 when the latter exists.  
Consider a skew field $\fF$, a gain graph $\Phi=(N,E,\phi)$ with gains in $\fF^*$, and the Desarguesian projective space $\bbP_N(\fF)$ coordinatized by $\fF$ with basis $\hN$.  The analytic representation begins with the apices:
\begin{align} \label{E:cevahyp}
\begin{aligned}
\be &= \hv \quad\text{ for a half edge at $v$,} \\
\be &= 
  \left.\begin{cases} 
\dfrac {1}{1+\phi(e_{vw})}(\phi(e_{vw}) \hv + \hw) &\text{if } \phi(e_{vw}) \neq -1,  \\[10pt] 
h_\infty \wedge \hv\hw &\text{if } \phi(e_{vw}) = -1
  \end{cases} \right\}\text{ for a link $e=e_{vw}$}
\end{aligned}
\end{align}
($\be$ was called $p^*(e)$ in Section IV.2.6), and then defines the hyperplane $\bh(e) := h(\be)$ by \eqref{E:cevianhyp} with $p:=\be$.  The resulting hyperplanes are the analytic Cevian representation $\cC[\Phi]$.  We quote the relevant part of the Generalized Theorem of Ceva from Part IV:

\begin{thm}[Theorem IV.2.18, first part] \label{IV.2.18}
The set of flats of the projective family $\cC(\Phi)$, ordered by reverse inclusion, is isomorphic to $\Lat G(\Phi)$ under the natural correspondence induced by $\bh$.
\end{thm}

\begin{thm}\label{T:cevequiv}
Consider a gain graph $\Phi=(N,E,\phi)$ with gains in the multiplicative group of a skew field $\fF$, and the Desarguesian projective space $\bbP_N(\fF)$ coordinatized by $\fF$.  
Then $\cC[\Phi]$ is a synthetic Cevian representation of $G(\Phi)$.  Stated precisely, the mappings $v \mapsto \hv^*$ and $e \mapsto h(e)$ are isomorphisms $\bgr{\Phi} \cong \Omega(\cC[\Phi])$ and $G(\Phi) \cong M(\cC[\Phi])$.

Conversely, given a biased graph $\Omega$, a synthetic Cevian representation of $\G(\Omega)$ in a projective geometry $\bbP(\fF)$ over a division ring $\fF$ is an analytic Cevian representation as in Equations \eqref{E:cevahyp}.
\end{thm}

\begin{proof}
The underlying graphs are obviously isomorphic under the specified mappings.  By Theorems \ref{T:cev} and \ref{IV.2.18}, the matroids are isomorphic.  Therefore the rank functions $\rk_{\Phi}$ and $\rk_{\Omega(\cC[\Phi])}$ correspond under the graph isomorphism so the biased graphs are also isomorphic.

It is clear that, given a synthetic Cevian representation $\cC$ of $\Omega$ based on the independent set $\hN \subseteq \bbP(\fF)$, one can use \eqref{E:cevahyp} in reverse to assign gains on $\Omega$ from $\fF^\times$ that make a gain graph $\Phi$ such that $\cC = \cC[\Phi]$ and $\bgr{\Phi}=\Omega$.
\end{proof}

\begin{ex}\label{X:cevianapices}
The apices of a Cevian representation are not a Menel{\ae}an representation.  Consider the representations in the plane over a skew field.  In the former the criterion for collinearity is a gain product of $-1$ (Menelaus' Theorem) and in the other there is concurrency when the gain product is $+1$ (Ceva's Theorem) not accepted.  The two criteria agree when the characteristic is $2$ but not otherwise.  
\end{ex}


\pagebreak[2]\section{Two dual representations of the lift matroid} \label{ortho}

Just as with the frame matroid, there are projectively dual representations of the lift matroid, one by points and the other by hyperplanes.  Here the most interesting viewpoints are affine instead of projective, which, if anything, makes the dual representations look more different from each other because the dual of the ideal hyperplane is a mere point.

\subsection{Points in parallel lines}\label{orthopar}\

\subsubsection{Projective orthography}\label{orthoproj}\

In a projective geometry $\bbP$ choose a base hyperplane $\bbP'$ and a projective representation, $\bz': E(\Delta) \to \bbP'$, of the graphic matroid $G(\Delta)$ of a nonempty simple graph $\Delta$ in $\bbP'$.  
(Since graphic matroids are regular, that is possible if and only if the dimension is large enough, i.e., $\dim\bbP \geq \rk\Delta = \#N-c(\Delta)$.)  
Also choose a point $\he_0$ in $\bbP \setminus \bbP'$.  The point determines a family of lines in $\bbP$: one line $\bz'(e)\he_0$ for each edge $e\in E(\Delta)$.  Call these lines the \emph{edge lines} in $\bbP$ and write $\hE(\Delta,\he_0)$ for the union of the edge lines.  
Given a subset $\hE$ of $\hE(\Delta,\he_0)\setminus \he_0$, we construct a biased graph $\Omega(\Delta,\bz',\hE,\he_0)$, abbreviated $\Omega_0(\hE)$, whose extended lift matroid is isomorphic to the projective dependence matroid $M(\hE \cup \{\he_0\})$; that is, $\hE \cup \{\he_0\}$ will be a representation of that extended lift matroid, with $e_0$ represented by $\he_0$.  

For the underlying graph $\Gamma$, let $N = N(\Delta)$ and let $E$ be a set in one-to-one correspondence with $\hE$ by $e \mapsto \he$.  (It is often convenient to take $E = \hE$.)  If $e\in E$, the \emph{projection} $p(e)$ is the edge of $\Delta$ whose edge line contains $\he$.
A \emph{cross-flat} is a projective flat that does not contain $\he_0$.  We define a circle $C$ in $\Gamma$ to be \emph{balanced} if $\hC$ spans a cross-flat.  Let $\cB$ be the class of balanced circles.  

\begin{thm}  \label{TD:orthorep}
This definition gives a biased graph $\Omega_0(\hE) = (\Gamma,\cB) = (N,E,\cB)$ whose extended lift matroid $L_0(\Omega_0(\hE))\cong M(\hE \cup \{\he_0\})$ by the correspondence $e \mapsto \he$ and $e_0 \mapsto \he_0$.  
\end{thm}

We say that $\hE$ is a (projective) \emph{synthetic orthographic representation} of a biased graph $\Omega$ if $\Omega \cong \Omega_0(\hE)$.  The representation mapping $E(\Omega) \to \hE$ is $e \mapsto \he$ when $e$ and $\he$ correspond to the same edge in $\Omega_0(\hE)$.

\begin{proof}  
We first prove a lemma.  Write $\rk_\Gamma$ for rank in $G(\Gamma)$.

\begin{lem}  \label{LD:orthorank}
For any $S\subseteq E$, 
\begin{equation*}
\rk(\hS) = \begin{cases}
\rk_\Gamma(S) 	&\text{if $S$ is balanced}, \\
\rk_\Gamma(S)+1 &\text{if $S$ is unbalanced.}
\end{cases}
\end{equation*}
\end{lem}

\begin{proof}  
It is clear that $\rk\big(\bz'(p(S)) \cup \{\he_0\}\big) = \rk \bz'(p(S)) + 1$ and that $\rk \bz'(p(S)) =\rk_\Delta p(S) = \rk_\Gamma S$.  It is also clear that $\Span (\hS \cup \{\he_0\}) = \Span\big(\bz'(p(S)) \cup \{\he_0\}\big)$.  Therefore,
$$
\rk\hS = \rk_\Gamma S + \epsilon,
$$
where $\epsilon = 0$ if $\he_0 \notin \Span\hS$ but $\epsilon = 1$ if $\he_0 \in \Span \hS$.  

If $S$ is unbalanced, it contains an unbalanced circle $C$ and $\he_0 \in \Span \hC$; then $\he_0 \in \Span\hS$ so $\epsilon = 1$. 

If $C$ is a balanced circle, then $\he_0 \not\in \Span \hC$; thus $\epsilon = 0$.  Thus Lemma \ref{LD:orthorank} is proved for a circle.  

Consider a forest $S$.  Since $\# S = \rk_\Gamma S$, we know that $\rk\hF$ must also equal $\rk_\Gamma S$.  Thus, $\hF$ is independent.  

Finally, consider a general balanced edge set $S$.  Take a maximal forest $F\subseteq S$.  For $e \in S \setminus F$, let $C_e$ be the fundamental circle of $e$ with respect to $F$.  Each $C_e$ is balanced; thus from the circle case and the independence of the forest $C \setminus e$ we infer $\he \in \Span(\hC \setminus \he)$.  Therefore $\hS \subseteq \Span \hF$, so $\rk \hS = \rk \hF = \rk_\Gamma S$.  In other words, $\epsilon = 0$.  This completes the proof of Lemma \ref{LD:orthorank}.
\end{proof}
  
Now we complete the proof of the theorem.  If we define $r(S) = \rk \hS$, then $r$ is a matroid rank function.  On the other hand, $r$ has the form of a lift-matroid rank function, in fact the one that is associated with $(\Gamma,\cB)$ if the latter is indeed a biased graph.  Proposition II.3.15 states that if this rank function defines a matroid, then $\cB$ is a linear class; hence, $\Omega_0(\hE)$ is a biased graph.
\end{proof}

\begin{cor}  \label{CD:orthobal}
An edge set $S\subseteq E$ is balanced in $\Omega_0(\hE)$ if and only if $\hS$ is contained in a cross-flat.  
$S$ is a closed, balanced set of $L(\Omega_0(\hE))$ if and only if $\hS = \hE \cap t$ for some cross-flat $t$.
\hfill $\square$
\end{cor}

Now recall the definitions of $\bbP$ and $\hE(\Delta,\he_0)$ from the beginning of this subsection.  We define an \emph{affine part} of a projective geometry to be what remains after deletion of any hyperplane.

\begin{prop}  \label{PD:orthoall}
Let $\Omega := \Omega_0(\hE(\Delta,\he_0))$.

{\rm (I)}  The biased graph $\Omega$ is a biased expansion of $\Delta$.  

{\rm (II)}  Suppose $\bbP$ has coordinate skew field $\fF$.  Assume $z'$ represents $\Delta$ in an affine part of $\bbP'$; if $\Delta$ is finite, this means $\Delta$ is $\#\fF$-colorable.  If $\bbP$ has coordinate skew field $\fF$ and $\Delta$ is $\#\fF$-colorable, then
$\Omega = \bgr{\fF^+ \Delta}$ i.e., $\Omega$ has gains in $\fF^+$.
\end{prop}

\begin{proof}  
(I) is a consequence of the definitions and Corollary \ref{CD:orthobal}.  

(II) follows by taking the deleted hyperplane $h_\infty'$ in $\bbP'$ and choosing the ideal hyperplane $h_\infty := h_\infty' \vee \he_0$ in $\bbP$; then $\he_0$ is an ideal point and $\bz'(E(\Delta))$ is disjoint from $h_\infty $.  Then choose coordinates for $\bbP$ so that $h_\infty$ is the infinite hyperplane in that coordinate system, in the affine part $\bbP\setminus h_\infty$ the hyperplane $\bbP'\setminus h_\infty'$ is a coordinate hyperplane, say for coordinate $z_0$, and in $\bbP'$ each $\bz'(e_{ij})=\hv_j-\hv_i$.  (This is possible because graphic matroid representation is projectively unique.)  For an edge $e$ of $\Omega$ that projects to $e_{ij}\in E(\Delta)$, the gain of $e$ (in the direction $v_iv_j$) is the $z_0$-coordinate of $\he$.

The colorability criterion comes from the Critical Theorem of Crapo and Rota \cite[Chapter 16]{CR}.   Assume $\bbP=\bbP(\fF)$.  According to the Critical Theorem, a finite graph $\Delta$ has a matroid representation in $\bbP'$ that avoids a hyperplane if and only if (i) $\bbP'$ is large enough, i.e., $\dim\bbP \geq \rk\Delta$, and (ii) $\fF$ is infinite or $\fF=\bbF_q$ and the chromatic polynomial of $\Delta$ satisfies $\chi_\Delta(q)>0$; equivalently, $\Delta$ is $\#\fF$-colorable.  This is true for every representation $\bz'$ and every hyperplane to be avoided.  (The Critical Theorem could be proved for non-Desarguesian projective geometries---all of which are planes, which unfortunately makes it less interesting---the proof by M\"obius inversion is the same and the order of $\fF$ is replaced by the order of the plane.)
\end{proof}

We can expand on this observation by considering a subset $\hE \subseteq \hE(\Delta,\he_0)$.  We call $\hE$ \emph{cross-closed with respect to $\Delta$} if, whenever $t$ is a cross-flat spanned by a subset of $\hE$, then $t \cap \hE(\Delta,\he_0)\subseteq \hE$.

\begin{cor}  \label{CD:orthocross}  
{\rm (I)}  Suppose $\hE\subseteq \hE(\Delta,\he_0)$.
Then $\hE$ is cross-closed if and only if $\Omega_0(\hE)$ is a biased expansion of a closed subgraph of $\Delta$.

{\rm (II)}  Suppose $\Delta$ is inseparable and is represented in an affine part of $\bbP'$, and suppose $\bbP$ has coordinates in a skew field $\fF$.  Then $\hE$ is cross-closed if and only if $\Omega_0(\hE) = \bgr{\Phi}$, where $\Phi\subseteq \fF^+ \Delta$ and $\Phi$ switches to a group expansion $\fH\Delta'$ of a closed subgraph $\Delta'$ of $\Delta$ by a subgroup $\fH\leq \fF^+$.
\end{cor}

\begin{proof}  
The proof is very similar to that of Corollary \ref{CD:mencross}.  The difference is that instead of edge lines $l_e = \hv\hw$ for $e_{vw} \in E(\Delta)$ that join node points we have edge lines $l_e=\he_0\bz'(e)$ that are projected (from $\he_0$) to edge points.  Nevertheless, the relationships between lines are similar, so we permit the reader to make the necessary adjustments.
\end{proof}

\subsubsection{Affine orthography}\label{orthoaff}\

There is an affine version of this representation.  Suppose $\bz'(E(\Delta))$ can be contained in an affine part of $\bbP'$, that is, $\bbA' = \bbP' \setminus h'$ for a hyperplane $h'$ of $\bbP'$.  
(By the Critical Theorem, mentioned in the proof of Proposition \ref{PD:orthoall}(II), if $\bbP=\bbP(\fF)$ and $\Delta$ is finite, this assumption is satisfied if and only if $\Delta$ is $\#\fF$-colorable.)  
Then we can treat $h_\infty := h' \vee \he_0$ as the ideal hyperplane, replacing $\bbP$ and $\bbP'$ by the affine spaces $\bbA = \bbP' \setminus h_\infty$ and $\bbA'$.  We may call $\hE$ an \emph{affine synthetic orthographic representation} of $\Omega\cong\Omega_0(\hE)$.  The now-ideal point $\he_0$ implies a direction in $\bbA$ transverse to $\bbA'$ and the edge lines are parallel lines in the direction of $\he_0$, one through $\bz'(e)$ for each edge $e\in E(\Delta)$.  A cross-flat is now an affine flat that does not projectively span $\he_0$; equivalently (if we restrict to cross-flats spanned by subsets of $\hE$), it is not parallel to any edge line.  Theorem \ref{TD:orthorep} remains true, of course, and so do Corollary \ref{CD:orthobal} and Proposition \ref{PD:orthoall} if $\bbA$ replaces $\bbP$ in the latter.

\subsubsection{Desarguesian orthography}\label{orthodes}\

In the case of a Desarguesian space $\bbP(\fF)$ our construction is equivalent to the analytic one in Section IV.4.1 by projection, a fact we now explain.  Analytically, for a gain graph $\Phi$ with gains in $\fF^+$ we represent the extended lift matroid $L_0(\Phi)$ in $\fF^{1+N} := \fF \times \fF^N$ by a mapping
\begin{equation}
\bz_\Phi(e) = \begin{cases} 
(\phi(e_{vw}), \hw - \hv) &\text{for a link } e_{vw}, \\
(1, \mathbf0) &\text {for a half edge or } e_0
\end{cases}
\label{E:liftrep}
\end{equation}
of $E(\Phi)$ into $\fF^{1+N}$.  
We call $\bz_\Phi$ an \emph{analytic orthographic representation of $\Omega$} if it represents an $\fF^+$-gain graph $\Phi$ such that $\bgr{\Phi} \cong \Omega$.  

Vector coordinates for $\fF^{1+N}$ are $(x_0,\bx')$; thus, $x_0 = \phi(e)$ or $1$ for $e \in E(\Phi) \cup \{e_0\}$.  Homogeneous coordinates are $[x_0,\bx']$ for the same point in the projective space $\bbP(\fF)$ that results from treating the lines of $\fF^{1+N}$ as points and vector subspaces as flats in the usual manner.
Thus $\bz_\Phi$ gives rise to a projective representation $\bar\bz_0$ of $L_0(\Phi)$.  
The content of the next theorem is that what we get is a synthetic representation and that all synthetic representations of biased graphs in $\bbP(\fF)$ arise in this way.  

\begin{thm}\label{T:orthoagrees}
In a Desarguesian geometry over $\fF$, the synthetic orthographic representations of extended lift matroids of biased graphs are the same as the projections from the origin of the analytic orthographic representations of the gain graphs $\Phi$ with gains in $\fF^+$.
\end{thm}

\begin{proof}
Suppose given an analytic representation $\bz_\Phi : E(\Phi) \to \fF^{1+N}$ with image $\hE$, based on $\bz': E(\Delta) \to \fF^N$.  We treat $\fF^N$ as the subspace $\{0\}\times\fF^N$ of $\fF^{1+N}$.  Projection from the origin turns $\fF^{1+N}$ into a projective space $\bbP$ with hyperplane $\bbP'$ as the image of $\fF^N$.  We use the notation $\bar v$ for the projective point derived from a vector $v$ and similar notation for mappings, i.e., $\bar\bz: E(\Phi) \to \bbP$ and $\bar\bz': E(\Delta) \to \bbP$.  For instance, the projective point $\bar\he_0$ is not in $\bbP'$ since the vector $\he_0=(1,\mathbf0) \notin \fF^N$, and $\bar\bz'(E(\Delta)) \subseteq \bbP'$ since $\bz'(E(\Delta)) \subseteq \fF^N$.

Under projection, linear dependence and independence become projective dependence and independence; therefore $\bar\bz'$ is a projective representation of $G(\Delta)$ and $\bar\bz_\Phi$ is a projective representation of $L_0(\Phi)$.  The plane spanned by $\he_0$ and the image of $p\inv(e)$ for an edge $e \in E(\Delta)$ becomes a line through $\bar\bz'(e)$ in the direction of $\bar\he_0$; it is the edge line $l_e$ of a synthetic affinographic representation.  Thus we have the entire structure of a synthetic representation of $L_0(\Phi)$.

Conversely, let $\bz: E(\Omega) \to \hE$ be a synthetic representation of $L_0(\Omega)$ in $\bbP$, a projective space over $\fF$, with associated representation $\bz': E(\Delta) \to \bbP'$.  We can assume, by enlarging $\bbP$ if necessary, that $\bbP'$ is large enough to contain an independent set $\hN := \{\hv : v \in N\}$ with respect to which $\bz'(E(\Delta))$ has standard coordinates (Lemma \ref{L:graphicrep} shows that and more), and we can assume $\bbP' = \Span(\hN\cup\{\he_0\})$ by restricting $\bbP$ to $\bbP' \vee \he_0$.  Since $\bbP$ is coordinatized and spanned by the independent set $\hN\cup\{\he_0\}$, we can assume the homogeneous coordinates of $\bbP$ represent $\he_0$ as $[1,\mathbf0]$.  As $\bbP$ is the quotient of $\fF^{1+N}$ under central projection, all the points of interest in $\bbP$, that is, $\hN, \hE, \he_0$, can be pulled back to vectors with the same linear dependencies as the projective versions have in $\bbP$.  By the choice of coordinate system, each point $\he$ representing an edge $e_{vw}$ of $\Omega$ has homogeneous coordinates $[x_0,\bw-\bv]$ and the vector has linear coordinates $(x_0,\bw-\bv)$.  Because we chose standard coordinates for $\bz'(E(\Delta))$ there is no ambiguity in the value of $x_0$ except that it could be replaced by its negative, which corresponds to reversing the direction of $e$.  Thus, if we assign $e$ the gain $\phi(e_{vw}) := x_0$ and, negatively, $\phi(e_{wv}) := -x_0$, we obtain an unambiguous $\fF^+$-gain graph $\Phi$ whose biased graph is $\Omega$ and whose analytic representation $\bz_\Phi$ in $\fF^{1+N}$ has the property that $\bar\bz_\Phi(e)$, the projectivization of $\bz_\Phi(e)$, equals $\he$.
\end{proof}

\subsection{Affinographic arrangements}\label{orthoaffino}\

\subsubsection{Affinographic hyperplanes}\label{affino}\

In an affine geometry $\bbA$, consider a family $\cA$ of hyperplanes that fall into parallel classes $\cA_i$ for $i \in I$, an index set.  In the projective completion $\bbP$ the hyperplanes of $\cA$ have ideal parts that are hyperplanes of the ideal hyperplane $h_\infty$; call this family of ideal hyperplanes $\cA_\infty$.  All the members of a parallel class $\cA_i$ have the same ideal part, call it $h_i$; thus $\cA_\infty = \{ h_i : i \in I \}$.  The family $\cA_\infty$ has a matroid structure $M(\cA_\infty)$ (it is the matroid of the dual points $h_i^*$ in the dual of the projective geometry $h_\infty$).  
If this matroid is graphic, isomorphic to $M(\Delta)$ for a graph $\Delta$, we call $\cA$ a \emph{synthetic affinographic hyperplane arrangement} (``arrangement'' here is a synonym for ``family'').  
The edges of $\Delta$ must be links, not half edges.  There may be more than one such $\Delta$; later we explain why that does not matter.  

In Section IV.4.1 we showed that the lift matroid of a gain graph $\Phi$ with gain group $\fF^+$ has a representation by a \emph{coordinatized affinographic hyperplane arrangement} $\cA[\Phi]$, whose hyperplane equations have the form $x_j-x_i = \phi(e_{ij})$.  In the present section we prove that synthetic affinographic arrangements are precisely the synthetic expression and generalization of those coordinatized affinographic arrangements.

First, given $\cA$, we construct a biased graph $\Omega(\cA)$ of which $\cA$ is a lift representation by hyperplanes; to say it precisely, such that $M(\cA_\bbP) = L_0(\Omega(\cA))$.  Here $\cA_\bbP$, the \emph{projectivization} of the affine arrangement $\cA$, is defined to be $\{ h_\bbP : h \in \cA\} \cup \{h_\infty\}$, $h_\bbP$ in turn being defined as the projective closure of $h$.  The underlying graph $\|\Omega(\cA)\|$ is constructed from a graph $\Delta$ such that $G(\Delta) \cong M(\cA_\infty)$.
$\Delta$ has one edge for each ideal part $h_i \in \cA_\infty$; we get $\|\Omega(\cA)\|$ if we replace each edge $e_i \in E(\Delta)$, corresponding to an ideal hyperplane $h_i$, by one edge for each $h \in \cA_i$.  Thus, if $\Delta = (N,E(\Delta))$, then $N(\Omega(\cA))=N$ and there is a natural surjection $p: E(\Omega(\cA)) \to E(\Delta)$, which is a graph homomorphism if we take $p|_N$ to be the identity.  

To define the class $\cB(\Omega(\cA))$ of balanced circles, we must know which edge sets are circles in $\Omega(\cA)$.  We get that information from $\Delta$: a circle in $\Omega(\cA)$ is either a digon $\{e,f\}$, where $p(e)=p(f)$, or a bijective preimage under $p$ of a circle in $\Delta$.  Define a digon in $\Omega(\cA)$ to be unbalanced, and let any other circle $C$ be balanced if and only if the intersection of its corresponding hyperplanes, $\bigcap_{e \in C} h_e$, is nonempty.  This defines $\cB(\Omega(\cA))$ so we have defined $\Omega(\cA)$, but we have not proved it is a biased graph.

\begin{thm}\label{T:affinobias}
Let $\cA$ be a synthetic affinographic arrangement of hyperplanes in an affine geometry $\bbA$.  Then $\Omega(\cA)$ is a biased graph, and $\cL(\cA) \cong \Latb\Omega(\cA)$ under the correspondence $h_e \mapsto e \in E(\Omega(\cA))$.  Furthermore, $\cL(\cA_\bbP) \cong \Lat L_0(\cA_\bbP)$ under the same correspondence augmented by $h_\infty \mapsto e_0$.  The intersection of $\cS \subseteq \cA$ is nonempty if and only if the corresponding edge set $S \subseteq E(\Omega(\cA))$ is balanced.
\end{thm}

We note that $\cL(\cA) \cong \Latb\Omega(\cA)$ is another way of asserting a matroid isomorphism $M(\cA_\bbP) \cong L_0(\Omega(\cA))$.

\begin{proof}
We write $E := E(\Omega(\cA))$ and $t_\bbA(S) := \bigcap_{e \in S} h_e$ if $S\subseteq E$.  Recall that the intersection poset of a hyperplane arrangement is partially ordered by reverse inclusion; $\cA_\infty$ is an upper interval in $\cA_\bbP$.  The cover relation is $x \gtrdot y$ ($x$ covers $y$).

The first step is to establish an isthmus reduction formula for an affinographic arrangement.  An \emph{isthmus} in a graph is an edge whose deletion converts one component into two.  

\begin{lem}[Isthmus Reduction]\label{L:affinoisth}
If $S \subseteq E$ has an isthmus $e$ and $t_\bbA(S\setminus e) \neq \eset$, then $t_\bbA(S) \neq \eset$ and $t_\bbA(S)$ is a hyperplane in $t_\bbA(S\setminus e)$.  
\end{lem}

\begin{proof}
We show that $t_\bbA(S) \neq \eset$.  The proof takes place in $\bbP$.  Define $t_\bbP(S) := \bigcap_{e \in S} (h_e)_\bbP$; if $t_\bbA(S) \neq \eset$ this is the projective completion of $t_\bbA(S)$.  Observe that $t_\infty(R)$ covers or equals $t_\bbP(R)$ for every edge set $R$.  Let $t_\infty(S) := t_\bbP(S) \cap h_\infty$.  Because $\cA_\infty$ is the graphic arrangement of $\Delta$ and $e \notin \clos_\Delta(S \setminus e)$, $t_\infty(S) \gtrdot t_\infty(S \setminus e)$ in $\cL(\cA_\infty)$.   As $(h_e)_\bbP$ is a hyperplane, $t_\bbP(S)$ covers or equals $t_\bbP(S \setminus e)$.  Because $t_\bbA(S \setminus e) \neq \eset$, $t_\infty(S \setminus e) > t_\bbP(S \setminus e)$.  We now have the Hasse diagram
$$
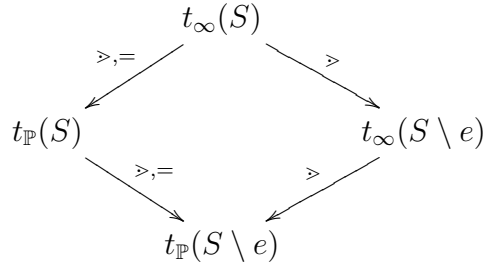
\xymatrix{
	&t_\infty(S) \ar[ld]_{\gtrdot,=} \ar[rd]^\gtrdot & \\
t_\bbP(S) \ar[rd]^{\gtrdot,=}	&	&t_\infty(S \setminus e) \ar[ld]_{\gtrdot} \\
	&t_\bbP(S \setminus e) & \\
}
$$
in which the path on the right has length 2.  The intersection lattice is geometric so the path on the left has the same length; it follows that $t_\infty(S) \gtrdot t_\bbP(S) \gtrdot t_\bbP(S \setminus e)$.  We conclude that $t_\bbP(S)$ is not contained in the ideal hyperplane, so $t_\bbA(S) \neq \eset$, and that $t_\bbP(S)$ is a hyperplane in $t_\bbP(S \setminus e)$.  It follows that $t_\bbA(S)$ is a hyperplane in $t_\bbA(S\setminus e)$.
\end{proof}

\begin{lem}\label{L:affinobal}
Let $C$ be a circle $C$ in $\Omega(\cA)$.  If $C$ is balanced, then $t_\bbA(C) = t_\bbA(C\setminus e) \neq \eset$ for each $e \in C$ and $\codim t_\bbA(C) = \#C-1$.  If $C$ is unbalanced, then $t_\bbA(C) = \eset$.
\end{lem}

\begin{proof}
We defined $C$ to be balanced if $t_\bbA(C)$ is nonempty, so the task is to prove that a balanced circle has $t_\bbA(C) = t_\bbA(C\setminus e)$ and $\codim t_\bbA(C) = \#C-1$.  Let $e \in C$.  
Since $\cA_\infty$ represents $G(\Delta)$, $t_\infty(C) = t_\infty(C \setminus e)$.  By isthmus reduction $t_\bbA(C \setminus e) \neq \eset$, so $t_\bbP(C \setminus e) \lessdot t_\infty(C \setminus e)$.  As $t_\infty(C) \geq t_\bbP(C) \geq t_\bbP(C \setminus e)$, $t_\bbP(C)$ equals either $t_\infty(C)$ or $t_\bbP(C \setminus e)$.  In the former case $t_\bbA(C) = \eset$.  In the latter case $t_\bbA(C) = t_\bbA(C \setminus e)$ and $\codim t_\bbA(C) = \#C-1$.
\end{proof}

The proof that $\cB(\Omega(\cA))$ is a linear class is also carried out in $\bbP$.  It is similar to that for Theorem \ref{T:cev}.

The next step is to prove that for a balanced edge set $S$, $t_\bbA(\bcl S) = t_\bbA(S) \neq \eset$.  Since $\bcl(\bcl S) = \bcl S$ if $S$ is balanced (Proposition I.3.5) and consequently $\bcl S = \bcl F$ for any maximal forest $F \subseteq S$, it suffices to prove that for a forest $F \subseteq E(\Omega)$, $t_\bbA(\bcl F) = t_\bbA(F) \neq \eset$.  An edge $e \in (\bcl F) \setminus F$ is in a balanced circle $C$ such that $C \setminus e \subseteq F$.  Since $t_\bbA(C) = t_\bbA(C \setminus e)$, we can infer that $t_\bbA(F \cup \{e\}) = t_\bbA(F)$.  It easily follows that $t_\bbA(\bcl F) = t_\bbA(F)$, using induction or Zorn's lemma according as $(\bcl F) \setminus F$ is finite or infinite.  

Thus the question reduces to proving that $t_\bbA(F) \neq \eset$.  We do so in the dual geometry $\bbP^*$.  
The hyperplanes $(h_e)_\bbP$ and $h_\infty$ become points $q_e$ and $q_\infty$; for $S \subseteq E$ let $S_\bbP := \{ q_e : e\in S\}$.  
The intersection $h_e \cap h_\infty$, which is a hyperplane of $h_\infty$, becomes the line $q_eq_\infty$.  
The arrangement $\cA_\infty$ becomes the point set $\bar E := \{\bar q_e := q_eq_\infty : e \in E\}$ in the projection $\bbP^*/q_\infty$ (that is the projective geometry of all lines through $q_\infty$ in $\bbP^*$).  Thus $M(\cA_\infty) = M(\bar E)$.  This matroid is $G(\Delta)$ under the correspondence $\bar q_e \leftrightarrow p(e) \in E(\Delta)$.  It is also the simplification of the contraction matroid $M(E_\bbP \cup \{q_\infty\})/q_\infty$.

The statement that $t_\bbA(F) = \eset$, which is equivalent to $\bigcap_{e \in F} (h_e)_\bbP \subseteq h_\infty$, becomes the statement that $F_\bbP$ spans $q_\infty$ in $\bbP^*$.  That means $F_\bbP \cup q_\infty$ is dependent in $M(F_\bbP \cup \{q_\infty\})$.
Now consider our hypothesis that $F$ is a forest.  The mapping $p|_F$ is injective into $E(\Delta)$ with image $p(F)$ that is a forest in $\Delta$.  Hence, $M(F_\bbP \cup \{q_\infty\})/q_\infty \cong G(\Delta)|p(F)$ (the vertical bar denotes a restriction matroid), which is a free matroid because $p(F)$ is a forest.  
However, the point set $\bar F$, as the contraction of $F_\bbP \cup \{q_\infty\}$, is dependent in $M(F_\bbP \cup \{q_\infty\})/q_\infty$.  This is a contradiction.  Therefore, $t_\bbA(F) \neq \eset$.

We now know that $t_\bbA(S) \neq \eset$ if $S$ is balanced. The inverse implication follows from the definition of balance in $\Omega(\cA)$.

It remains to prove that $\cL(\cA_\bbP) \cong \Lat L_0(\Omega(\cA))$, and that $cL(\cA) \cong \Latb \Omega(\cA)$.  The latter is an immediate corollary of Lemma \ref{L:affinobal}, since a closed, balanced set in $\Omega(\cA)$ is the same as a set $S$ such that $t_\bbA(S)\neq\eset$ and $t_\bbA(S\cup \{e\}) \subset t_\bbA(S)$ for any $e \notin S$.  

The lattice $\Lat L_0(\Omega(\cA))$ has a particular structure (see Theorem II.3.1(A)(b)), namely, $\Latb \Omega(\cA)$ is a lower ideal and its complement is essentially $\Lat \|\Omega(\cA)\|$, which is isomorphic to $\Lat \Delta$ since $\Delta$ is the simplification of $\|\Omega(\cA)\|$.  The precise relationship is that each closed set $S \in \Lat \|\Omega(\cA)\|$ appears in $\Lat L_0(\Omega(\cA))$ as the set $S \cup \{e_0\}$.  As $e_0$ corresponds to $h_\infty$, this upper ideal of $\Lat L_0(\Omega(\cA))$ corresponds to $\cL(\cA_\infty)$, whose structure is that of $\Lat \Delta$ by assumption.  It is easy to verify that the order relations in $\cL(\cA_\bbP)$ agree with those in $\Lat L_0(\Omega(\cA))$.  
That completes the proof.
\end{proof}

The non-uniqueness of $\Delta$ calls for some discussion.  Whitney proved that two graphs of finite order have the same matroid if and only if one can be obtained from the other by three operations, each of which preserves the edge sets of circles (though not necessarily their cyclic order in the circle).  Whitney's \emph{2-operations} are:
\begin{enumerate}[(Wh1)]
\item Combine two components by identifying one node in the first component with one node in the other component.  (The identified node will be a separating node.)
\item The reverse of (Wh1), i.e., splitting a component in two at a separating node.
\item (\emph{Whitney twist}.)  Reversing the attachment of one side of a 2-separation.  Suppose $\{u,v\}$ is a pair of nodes that separates a subgraph $\Delta_1$ from another subgraph $\Delta_2$.  Split $u$ and $v$ into $u_1,v_1$ in $\Delta_1$ and $u_2,v_2$ in $\Delta_2$.  Then recombine by identifying $u_2$ with $v_1$ and $v_2$ with $u_1$.
\end{enumerate}
Graphs related by Whitney's 2-operations are called \emph{2-isomorphic}.  
It is clear from Whitney's proof that the same theorem applies to graphs of infinite order.  
So, $\Delta$ is determined precisely up to Whitney 2-isomorphism.  It is also easy to see that Whitney's theorem applies to the extended lift matroid; thus, $\Omega(\cA)$ is determined only up to 2-isomorphism; but once $\Delta$ has been chosen, $\Omega(\cA)$ and $L_0(\Omega(\cA))$ are completely determined.  

We wish to establish the relationship between synthetic and coordinatized affinographic arrangements; we begin by reviewing the latter in some detail (from Section IV.4.1).
Given $\fF$, suppose we have a gain graph $\Phi=(N,E,\phi)$ (without half edges; $N$ and $E$ may be infinite) with gain group $\fF^+$.  This gain graph has a representation in the $\#N$-dimensional affine space $\bbA^N(\fF)$ by hyperplanes $h_e: x_j-x_i=\phi(e_{ij})$ for each edge $e = e_{ij} \in E$.  Write $\cA[\Phi]$ for this family of hyperplanes; we call it an \emph{affinographic hyperplane arrangement} (``coordinatized'' to distinguish it from the synthetic analog) because it consists of affine translates of graphic hyperplanes $h_{ij}: x_j-x_i=0$.  
The hyperplanes of a parallel class of edges, i.e., all edges $e_{ij}$ for fixed nodes $v_i,v_j$, are a parallel class of hyperplanes whose intersection in the ideal part of $\bbP^N(\fF)$ (the projective completion of $\bbA^N(\fF)$) is a relative hyperplane $h_{ij\infty}$.  These ideal parts $h_{ij\infty}$ are a graphic arrangement associated with $\Delta$, the simplification of $\|\Phi\| := (N,E)$; specifically, they are the ideal parts of the standard graphic arrangement $\cH[\Delta] = \{ h_{ij} : v_iv_j \in E(\Delta) \}$ in $\bbA^N(\fF)$.  
The essential property of $\cA[\Phi]$ for our purposes is that $\cL(\cA[\Phi]) \cong \Latb\Phi$ under the natural correspondence $h_e \leftrightarrow e$; in addition there are the projectivization $\cA_\bbP[\Phi]$ and the ideal part $\cA_\infty[\Phi]$, which satisfy $\cL(\cA_\bbP[\Phi]) \cong \Lat L_0(\Phi)$ and $\cL(\cA_\infty[\Phi]) \cong \Lat G(\|\Phi\|) \cong \Lat G(\Delta)$.  (The latter two isomorphisms can also be expressed as matroid isomorphisms.)

\begin{thm}\label{T:affinodes}
In an affine geometry $\bbA(\fF)$ over a skew field $\fF$, synthetic affinographic hyperplane arrangements in $\bbA(\fF)$ are the same as coordinatized affinographic arrangements.
\end{thm}

We mean that, given a synthetic affinographic arrangement $\cA$, we can choose the coordinate system so as to express $\cA$ by affinographic equations.  The construction of suitable coordinates will be part of the proof.

One more remark before the proof.  Since every affine geometry beyond planes is coordinatized by a skew field, what is the purpose of Theorem \ref{T:affinodes}?  There are several answers.  First, it is not \emph{a priori} true that the synthetic definition has no examples that do not coordinatize.  A more fundamental answer is that by defining synthetic affinographic arrangements we are axiomatizing affinographic arrangements.  And there is a third answer:  the theorem lets us treat affinographic arrangements of lines in a non-Desarguesian plane; for that see \cite[\qlift]{BGPP}.

\begin{proof}
We already explained how a coordinatized affinographic arrangement is a synthetic affinographic arrangement.  
For the converse let $\cA$ be a synthetic affinographic arrangement in an affine geometry $\bbA(\fF)$ over $\fF$; we construct coordinates that put $\cA$ in coordinatized form.

A \emph{standard representation} of a graphic matroid $G(\Delta)$ is a representation in $\fF^N$, with unit basis vectors $\hv$ for $v \in N$, such that an edge $e_{v_iv_j}$ is represented by $\he = \hv_j-\hv_i$ (or its negative).  This representation is in a vector space; a standard representation in an affine geometry $\bbA$ is obtained by choosing an origin $o$, so $\bbA$ becomes a vector space, and forming a standard representation in that vector space.

\begin{lem}\label{L:graphicrep}
Given a skew field $\fF$, a possibly infinite simple graph $\Delta = (N,E(\Delta))$, and an embedding $e \mapsto \be$ of the matroid $G(\Delta)$ in the ideal  hyperplane $h_\infty$ of a projective space $\bbP(\fF)$, there are an origin $o \in \bbA := \bbP(\fF) \setminus h_\infty$ and a standard representation $e \mapsto \he$ of $G(\Delta)$ in $\bbA$ such that $o\he \wedge h_\infty = \be$.
\end{lem}

In other words, $\bE$ should be the projection of $\hE$ into $h_\infty$ from center $o$.

\begin{proof}
We can treat each component of $\Delta$ separately, so assume it is connected.  Let $T$ be the edge set of a spanning tree with root node $v_0$ and set $N_i :=\{ v \in N : \text{distance}_T(v_0,v) = i \}$.  Enlarge $\bbP(\fF)$ if necessary so there is a point $\bv_0$ in $h_\infty \setminus \Span \bE$.  Finally, choose a point $o \in \bbA$ to serve as an origin; that makes $\bbA$ into a vector space $\bbA_o$ over $\fF$ with addition defined by the parallelogram law.  

Now we inductively construct a basis $\{ \hv : v \in N \}$ for a subspace of $\bbA$ and representing points $\he$, $e \in E(\Delta)$, such that $\he = \hv_j - \hv_i$ for each edge $e_{v_iv_j} \in E(\Delta)$.  The first step is to choose $\hv_0 \in o\bv_0 \setminus \{o,\bv_0\}$.  Then for each edge $e_{v_0v_1} \in T$ such that $v_1 \in N_1$, choose $\he \in o\be \setminus \{o,\be\}$ and define $\hv_1 := \hv_0 + \he$; thus $\he = \hv_1-\hv_0$.  At the $i$th step, $\hv_{i-1}$ has been chosen for all nodes $v_{i-1} \in N_{i-1}$.  For every edge $e_{v_{i-1}v_i} \in T$ such that $v_{i-1}\in N_{i-1}$ and $v_i \in N_i$, choose $\he \in o\be \setminus \{o,\be\}$ and define $\hv_i := \hv_{i-1} + \he$; thus $\he = \hv_i-\hv_{i-1}$.  (If there are infinitely many edges of the form $e_{v_{i-1}v_i} \in T$ we apply the Axiom of Choice.)  Induction extends the definitions of $\hv$ and $\he$ to all nodes and to all edges in $T$.  Note that in the algebraic formula $\he = \hv_i-\hv_{i-1}$ we treat $e$ as oriented from $v_{i-1}$ to $v_i$.  If we reverse the orientation, the representing vector of $e$ is negated to $\hv_{i-1}-\hv_i$.

For an edge $e_{v_iv_j} \notin T$ there is a path $P=f_1\cdots f_l \subseteq T$ joining $v_i$ and $v_j$.  Orient each $f_j \in P$ in the direction from $v_i$ to $v_j$ and define $\he := \hf_1+\cdots+\hf_l$ with $\hf_j$ defined according to the orientation of $f$.  Then $\he = \hv_j-\hv_i$.  That completes the construction of a standard representation of $G(\Delta)$ in $\bbA_o$, since the linear dependencies of the vectors $\he$ are the same as the projective dependencies of the ideal points $\be$.  That completes the proof.

If $\hN$ does not span $\bbA_o$, we may extend it to a basis of $\bbA_o$ in order to provide a complete coordinate system.  However, that is not required for the proof.
\end{proof}

To prove Theorem \ref{T:affinodes} we dualize Lemma \ref{L:graphicrep}.  The dual     lemma says that a hyperplanar embedding $\cA_\infty$ of $G(\Delta)$ in $h_\infty$ extends to a hyperplanar representation $\cH[\Delta]$ in $\bbP(\fF)$ in which $e$ corresponds to an affine hyperplane $h_e$ whose ideal part $(h_e)_\infty$ is the given representative of $e$ in $h_\infty$, and whose hyperplanes have equations $x_i=x_j$ in a coordinate system for $\bbA(\fF)$.  
A hyperplane in the synthetic affinographic arrangement $\cA$ belongs to a parallel class whose ideal part $h_k$ is an element of $\cA_\infty$ corresponding, say, to $e \in E(\Delta)$; therefore it is parallel to one of the hyperplanes of $\cH[\Delta]$ and has equation of the form $x_j-x_i=c$, a constant.  Then we put an edge $e_{ij}$ in $\Phi$ with gain $\phi(e_{ij}) = c$.  Taking $N(\Phi) := N$, this defines the right gain graph for $\cA$, i.e., $\cA[\Phi] = \cA$.  Producing $\Phi$ completes the proof of the theorem.
\end{proof}

\subsubsection{Expansions}\label{affinoexpan}\

We turn now to the example of group and biased expansions.  Take a subgroup $\fH$ of $\fF^+$; the group expansion $\Phi = \fH\Delta$ gives an affinographic arrangement $\cA[\Phi]$ in $\fF^N$.  Viewed abstractly, we have an affinographic hyperplane representation of the biased graph $\bgr{\fH\Delta}$.  Now we reverse the process.  

\begin{cor}\label{C:affinoexpansion}
Suppose $\cA$ is a synthetic affinographic hyperplane arrangement that is a lift representation of a biased expansion graph $\Omega\downa\Delta$ in an affine geometry $\bbA(\fF)$ over a skew field $\fF$.  Assume $\Delta$ is simple, inseparable, and of order at least $3$.  Then there are a subgroup $\fH \subseteq \fF^+$ and an $\fF^+$-gain graph $\Phi$ such that $\Phi$ switches to $\fH\Delta$ and, with suitable coordinates in $\bbA(\fF)$, $\cA = \cA[\Phi]$.
\end{cor}

\begin{proof}
Theorem \ref{T:affinodes} implies that $\cA = \cA[\Phi]$ for some gain graph $\Phi$ with gains in $\fF^+$.  Since $\cA$ is a lift representation of a biased expansion of $\Delta$, $\|\Phi\|$ must simplify to $\Delta$.  We want to prove $\Phi$ switches to a subgroup expansion of $\Delta$.  If $\Omega$ is a trivial expansion, the subgroup is trivial.  Otherwise, $\Omega$ is thick and the representation is canonical (Theorem \ref{T:thick}), which implies that $\Omega$ has gains in $\fF^+$ (Theorem \ref{T:canonicalrep}).  A biased expansion with gains in a group $\fG$ has to switch to a subgroup expansion (Lemma \ref{L:subgroupexp}).
\end{proof}

For the corresponding result about non-Desarguesian planes see \cite[\Paffsubnet]{BGPP}.

\subsubsection{Duality}\label{orthodual}\

Finally, we briefly explain how the point and hyperplanar lift-matroid representations are dual to each other.  
Summarizing the dual correspondence: the point representation $e \mapsto \he$ of $G(\Delta)$ in $\bbP'$ dualizes to the hyperplanar representation $e \mapsto (h_e)_\bbP$ in $\bbP$; the extra point $\he_0$ corresponds to $h_\infty$; the edge line $l_{ij} = \bz'(v_iv_j)\vee\he_0$ for $v_iv_j \in E(\Delta)$ corresponds to $(h_e)_\infty = (h_e)_\bbP \wedge h_\infty = h_{ij}$; and, last but far from least, the representing point $\he$ of an edge $e \in E$ is dual to the affine or projective hyperplane $h_e$ or $(h_e)_\bbP$.

We did not use this fact to prove anything about them, for three reasons.  First, we think it is good to have direct proofs, thereby illustrating how one works with the representations.  Second, we found the projective interpretation of the orthographic representations was stronger but the affine interpretation of affinographic hyperplanes is the more interesting.  Third, the reader will have noticed that the proofs share a dependence on Lemma \ref{L:graphicrep}, but less so for the orthographic than for the affinographic representations, a fact that suggests the separate proofs are not themselves mere duals of each other.  Furthermore, since the corollaries and developments are not precisely parallel, by dualizing some results one can get a few more properties that we did not mention.


\pagebreak[2]

\end{document}